\documentclass{amsart}

\usepackage{amssymb,tikz,pgfplots}
\usepackage{amsthm}
\usepackage{amsmath, a4wide}
\usepackage{amsbsy}
\usepackage[all]{xy}
\usepackage{bm,color}
\usepackage{hyperref,caption}

\newcommand{\RR}{\mathbb{R}}

\newcommand{\ud}{\textup{d}}

\DeclareMathOperator{\wind}{wind}

\begin{document}

\newtheorem{theorem}{Theorem}
\newtheorem*{proposition}{Proposition}
\newtheorem*{thm}{Theorem}
\newtheorem*{corollary}{Corollary}
\newtheorem*{conjecture}{Conjecture}
\newtheorem{lemma}{Lemma}

\title[Carrier frequencies, holomorphy and unwinding]{Carrier frequencies, holomorphy and unwinding}

\author{Ronald R. Coifman, Stefan Steinerberger \and Hau-tieng Wu}
\address{Department of Mathematics, Program in Applied Mathematics, Yale University, 51 Prospect Street, CT 06511, USA}
\email{coifman@math.yale.edu}

\address{Department of Mathematics, Yale University, 10 Hillhouse Avenue, CT 06511, USA}
\email{stefan.steinerberger@yale.edu}

\address{Department of Mathematics, University of Toronto, 40 St George Street, Toronto, ON M5S 2E4, Canada}
\email{hauwu@math.toronto.edu}

\maketitle
 
\vspace{-20pt}

\begin{abstract} We prove that functions of intrinsic-mode type (a classical models for signals) behave essentially like holomorphic functions: adding a pure carrier frequency $e^{int}$
ensures  that the anti-holomorphic part is much smaller than the holomorphic part
$ \| P_{-}(f)\|_{L^2} \ll \|P_{+}(f)\|_{L^2}.$
This enables us to use techniques from complex analysis, in particular the \textit{unwinding series}. We study its stability and convergence properties and show that the unwinding series can stabilize and show that the unwinding series can provide a high resolution time-frequency representation, which is robust to noise. 
\end{abstract}

\section{Introduction} 

\subsection{Introduction}
Time-frequency analysis is at the very core of signal processing and widely used in
applications (MP3 audio, JPEG images, ...). The simplest possible example of a 'signal'
is certainly a cosine polynomial, i.e. a superposition of stationary signals
$$ f(t) = a_0 + \sum_{k=1}^{K}{a_k \cos{(2\pi k t)}}$$
and classical Fourier analysis allows for the analysis of such signals. However, in real-life
applications this assumption on the form of the signal is overly restrictive and both frequencies
and amplitudes can slowly shift over time. This notion is usually either attributed to Gabor \cite{Gabor:1946}
or van der Pol \cite{vanderPol:1946}. Given a real signal $G(t)$, Gabor introduced the complex signal extension
$$ G^c = G + i \mathcal{H} G, \qquad \mbox{where $\mathcal{H}$ is the Hilbert transform}.$$
This complex signal can now be written in polar coordinates as
$$ G^c(t) = a(t) e^{i \phi(t)},$$
where $a(t)\geq 0$ and $\phi'(t)$ are the natural quantities of interest, called the \textit{amplitude modulation} (AM) and the \textit{instantaneous frequency} (IF). Ultimately,
we are interested in obtaining stable ways of estimating the instantaneous frequency $\phi'(t)$ for a given signal.

\subsection{Related work}

In the past decays, several approaches are proposed to deal with this problem. These approaches could be roughly classified into two categories -- one is decomposing the signal into oscillatory ingredients first and extract the amplitude modulation (AM) and intrinsic frequency (IF) information while the other one tries to obtain the time-frequency (TF) representation. In the first category, examples include the empirical mode decomposition (EMD) \cite{Huang_Shen_Long_Wu_Shih_Zheng_Yen_Tung_Liu:1998} and its variations like \cite{Dragomiretskiy_Zosso:2014,Gilles:2013,Pustelnik_Borgnat_Flandrin:2014,Wu_Huang:2009}, the sparsity approach \cite{Tavallali_Hou_Shi:2014}, the iterative convolution-filtering approach \cite{Cicone_Liu_Zhou:2016,Lin_Wang_Zhou:2009}, the approximation approach \cite{Chui_Mhaskar:2016}, etc. The oscillatory components that arise out of the decomposition are often called the ``intrinsic mode functions''. If the IF and AM information is not obtained while decomposing the signal, a different transform, like the Hilbert transform, is needed to estimate the IF and AM from the intrinsic mode functions. In the second category, we find the well-known linear TF analysis, like short time Fourier transform (STFT) \cite{Flandrin:1999,Flandrin:2015}, continuous wavelet transform (CWT) \cite{Daubechies:1992}, Chirplet transform \cite{Mann_Haykin:1995}, S-transform \cite{Stockwell_Mansinha_Lowe:1996}, etc, the quadratic TF analysis, like the Wigner-Ville distribution and its generalization like Cohen's class or Affine class \cite{Flandrin:1999}, and the nonlinear variation of these methods like the reassignment method and its variation \cite{Auger_Chassande-Mottin_Flandrin:2012,Auger_Flandrin:1995,Kodera_Gendrin_Villedary:1978}, the synchrosqueezing transform (SST) \cite{Daubechies_Lu_Wu:2011,Daubechies_Maes:1996}, the concentration of frequency and time (ConceFT) \cite{Daubechies_Wang_Wu:2016}, the scatter transform \cite{Mallat:2012}, the variation of the Gabor transform \cite{Balazs_Dorfler_Jaillet_Holighaus_Velasco:2011,Galiano_Velasco:2014,Ricaud_Stempfel_Torresani:2014}, the cepstrum-based approach \cite{Lin_Su_Wu:2016}, etc. After obtaining the time-frequency representation of the signal, the IF and AM can be extracted.

\section{IMT functions are almost holomorphic}
\subsection{Intrinsic mode type function.} It is obvious that in order for the problem to be well-posed, one needs to put some additional restriction on
the various components (otherwise the reconstruction problem has too many degrees of freedom). A classical
approach is restricting every single component to be close to an intrinsic-mode function. The following definition
is the natural adaption of \cite[Definition 3.1.]{Daubechies_Lu_Wu:2011} to the periodic setting.
\begin{quote}
\textbf{Definition.} A periodic, continuous function $f:\mathbb{R} \rightarrow \mathbb{C}$ is said to be an \textit{intrinsic-mode type} (IMT)
with accuracy $\varepsilon > 0$ if $f(t) = A(t) e^{i \phi(t)}$ with
\begin{align*}
&A \in C^1(\mathbb{T}, \mathbb{R}_{+}), \quad \phi \in C^2(\mathbb{T}, \mathbb{T}) \\
&\inf_{t \in \mathbb{R}}{ \phi'(t)} >0 \qquad \sup_{t \in \mathbb{R}}{\phi'(t)} < \infty \\
& |A'(t)| \leq \varepsilon \phi'(t) \qquad \quad |\phi''(t)| \leq \varepsilon \phi'(t)  
\end{align*}
\end{quote}
Geometrically, the function $f: \mathbb{T} \rightarrow \mathbb{C}$ winds counter-clockwise around the origin and is only able
to change its distance to the origin slowly compared to its instantaneous frequency $\phi'(t)$.
This setup is natural for a variety of different reasons; it would, for example, be impossible to get a decent understanding of the
intrinsic frequency $\phi'$ using merely a finite number of samples unless there was some control on the behavior of $\phi'$ between
samples (here exerted by $|\phi''| \leq \varepsilon |\phi'|$).\\

\begin{center}
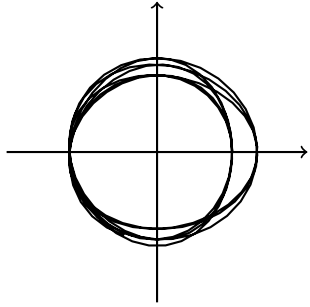
\begin{figure}[h!]
\begin{tikzpicture}
\draw [->, thick] (-2,0) -- (2,0);
\draw [->, thick] (0,-2) -- (0,2);
\draw [black, thick,  domain=-0:24.1, samples=200] plot ({(1+sin(25*\x)^2/3)*cos(100*\x)},{(1+sin(20*\x)^2/4)*sin(100*\x)});
\end{tikzpicture}
\caption{$f(e^{it})$ for a function of intrinsic-mode type: amplitude changes slowly with respect to phase.}
\end{figure}
\end{center}

\subsection{Main result.} Our contribution is an elementary estimate that guarantees that functions of intrinsic-mode type are very close 
to being holomorphic. We quantify this notion using the Littlewood-Paley projections $P_{-}, P_{+}$ onto negative and nonnegative 
frequencies, respectively. This allows us to decompose any function into
$$ f = P_{+}f + P_{-}f \qquad \mbox{a holomorphic and an anti-holomorphic part.}$$
Our main statement says that under the assumptions above, only a small part of the function can be anti-holomorphic (in particular,
projection onto holomorphic function has very little impact).

\begin{theorem}[Intrinsic mode functions are almost holomorphic]\label{Theorem:CarrierPhase} Let $f = A(t)e^{i \phi(t)}$ be an intrinsic-mode function with accuracy
$\varepsilon$. Then 
$$\| P_{-} f\|_{L^2}^2  \leq \left( \frac{8 \pi^2}{\|A\|_{L^2}^2} \frac{\|A'\|_{L^{\infty}}^2 + \varepsilon^2  \|A\|_{L^{\infty}}^2}{\inf_{0 < t < 2\pi}{\phi'(t)}}   \right) \|f\|_{L^2}^2$$
\end{theorem}

The statement by itself might not seem very useful as, in general, one may not have any explicit control on these quantities and can
thus not guarantee that it is small. However, in the usual setting of $\phi'(t)$ being uniformly large and the amplitude only undergoing
slow changes, the bound indeed states that any function with small variation in the amplitude that moves counterclockwise 
around the origin sufficiently fast is close to being holomorphic.

\begin{theorem}[Small error in the phase] Given a signal $f(t) = A(t) e^{i \phi(t)}$, we use $\phi^*$ to denote the phase of its holomorphic projection
$$ P_{+} (A(t) e^{i \phi(t)}) = |P_{+} (A(t) e^{i \phi(t)})| e^{i \phi^*(t)}.$$
 Then we can control the error
$$ \| \phi(t) - \phi^*(t) \|^2_{L^2} \leq  \left( \frac{8 \pi^4}{\|A\|_{L^2}^2} \frac{\|A'\|_{L^{\infty}}^2 + \varepsilon^2 \|A\|_{L^{\infty}}^2}{\inf_{0 < t < 2\pi}{|A(t)|^2}} \frac{1}{\inf_{0 < t < 2\pi}{\phi'(t)}}  \right) \|f\|_{L^2}^2.$$
\end{theorem}
The very important consequence is that if we decide to add a carrier frequency and replace 
$$ \mbox{the signal}~ A(t)e^{i \phi(t)} \qquad \mbox{by the new signal} \qquad A(t) e^{i \phi(t)} e^{i N t}$$
for some $N \in \mathbb{N}$, then we can immediately deduce the phase of one from the other (by subtracting $Nt$), however,
\textit{the bound in the corollary scales in our favor}: the amplitude function $A(t)$ does not change at all but $\inf_{0 < t < 2\pi}{\phi'(t)}$
increases by at least $N$, which guarantees that the error we make is smaller.
Indeed, the error will tend to 0 as $N \rightarrow \infty$. It is not difficult to see that
$$\forall f \in L^2 \qquad \qquad \lim_{N \rightarrow \infty}{P_{+}(e^{i N t} f(t))e^{- i N t}} = f$$
but this convergence \textit{is not uniform}. Our main statement guarantees that there is some uniform control within
the class of functions of IMT type using precisely those quantities that are being used to classify IMT functions.

\subsection{Adding carrier frequencies.} 
 Therefore, given any method that requires signals to be 
holomorphic for them to be analyzed, we may proceed as follows:
\begin{enumerate}
\item Let $f(t)= A(t)e^{i \phi(t)} $ be a signal to be analyzed.
\item Add a carrier frequency $A(t) e^{i \phi(t)} e^{i N t}$.
\item Project that function onto holomorphic functions $P_{+} (A(t) e^{i \phi(t)} e^{i N t})$.
\item Find the phase $\phi^*(t)$ of the holomorphic function $P_{+} (A(t) e^{i \phi(t)} e^{i N t})$.
\item Use $\phi^*(t) - Nt$ as approximation of the phase $\phi(t)$.
\end{enumerate}
In particular, as $N$ increases the desired function moves closer to the subspace of holomorphic functions and the projection in
step (3) has less of an impact. Therefore, one can (at least in theory) assume any function to be holomorphic up to an arbitrarily small error. This pre-processing technique is, of course, not restricted to particular applications but may prove advantageous for a variety of
different techniques: further below, we give numerical examples showing its effect on the synchrosqueezing transform (see also \cite{Xi_Cao_Chen_Zhang_Jin:2015}).

\section{Nonlinear phase unwinding via Blaschke series}

\subsection{The unwinding series.} These results allow us to apply purely complex analysis methods that require holomorphic input to arbitrary IMT signals after adding a suitable carrier signal. We believe that one of most natural ways
of analyzing complex signals is the \textit{unwinding series}. If $F:\mathbb{C} \rightarrow \mathbb{C}$ is a holomorphic signal (or the complexification of a real signal in the sense of Gabor), then we have classical Fourier series
at our disposal: one particularly simple derivation is based on deriving them from de Moivre's identity and a power series expansion. Trivially,
$$ F(z) = F(0) + (F(z) - F(0)).$$
Since $F(z) - F(0)$ vanishes in 0, it has a root there and we can write it as $F(z) - F(0) = z F_1(z)$ for another holomorphic $F_1$. Reiterating the procedure, we get
\begin{align*}
F(z) &= F(0) + (F(z) - F(0)) \\
&= F(0) + z F_1(z) \\
&= F(0) + z F_1(0) + z^2 F_2(0) + z^3 F_3(0) + \dots
\end{align*}
Setting now $z = e^{it}$, we have found that
$$ F( e^{it} ) =\sum_{k=0}^{\infty}{a_k e^{ikt}}.$$
 In the mid 1990s, the first author proposed the following modification: instead of merely factoring out the root at the origin, one could just as well factor out 
all the roots inside the unit disk $\mathbb{D} \subset \mathbb{C}$. The arising factors, \textit{Blaschke products}, are by now classical objects in complex analysis and can be written as
$$ B(z) = z^m \prod_{k}{ \frac{\overline{a_k}}{|a_k|}  \frac{z-\alpha_k}{1-\overline{\alpha_k}z}} \qquad \alpha_k \in \mathbb{D}.$$
The crucial ingredient allowing for factorization is the following classical theorem (established at different levels of regularity for $F$ which we skip for brevity).
\begin{thm}[Blaschke factorization, see e.g. Garnett \cite{Garnett:1981}]
A holomorphic $F:\mathbb{C} \rightarrow \mathbb{C}$ can be written as
$$ F = B \cdot G \qquad \mbox{where}~B~\mbox{is a Blaschke product}$$
and $G$ has no roots inside the unit disk $\mathbb{D}$.
\end{thm}
Formally, an application of this fact repeatedly allows us again to write a function as
\begin{align*}
F(z) &= F(0) + (F(z) - F(0)) \\
&= F(0) + B_1(z) G_1(z) \\
&= F(0) + B_1(z) (G_1(0) + (G_1(z)-G_1(0))) \\
&= F(0) + G_1(0) B_1(z) + G_2(0) B_1(z) B_2(z) + G_3(0) B_1(z) B_2(z) B_3(z) + \dots
\end{align*}
First numerical experiments were carried out in the PhD thesis of Michel Nahon \cite{Nahon:2000Thesis} with further contributions by Letelier \& Saito \cite{Saito_Letelier:2009} and Healy \cite{ Healy:2009b, Healy:2009a}.
The first rigorous proof of convergence was carried out by T. Qian \cite{Qian:2010} for initial data in the Hardy space $\mathcal{H}^2$.
The first two authors \cite{Coifman_Steinerberger:2015} gave a wide range of convergence results: in particular, for all $s > -1/2$ the sequence converges in the Sobolev space $H^{s}(\mathbb{T})$ for initial data
in the Dirichlet space $H^{s+1/2}(\mathbb{T})$.

\subsection{Stability}  
The purpose of this section is to describe a stability result for Blaschke decompositon under white noise. We shall assume that we are given a holomorphic
signal $F(e^{it})$ on the boundary of the unit disk and, by an abuse of notation, can write the Blaschke decomposition as 
$$ F(z) = B(z) \cdot G(z) \qquad \mbox{inside} \quad \mathbb{D}.$$
We will now suppose that $F(e^{it})$ is perturbed by white noise $\Phi$ and that we are given $F(e^{it}) + \Phi$ instead. Denoting the Poisson extension by $\mathcal{P}$
$$ (F + \mathcal{P}\Phi)(z) = B_1(z) \cdot G_1(z),$$
we will investigate how $B$ and $B_1$ are related. Clearly, since $G$ and $G_1$ are outer functions and Blaschke products are determined by their roots, this
amounts to understanding how the roots of 
$$ F(z) \quad \mbox{inside}~\mathbb{D}~\mbox{are related to the (random) roots of}~ (F + \mathcal{P}\Phi)(z).$$
We will now show that $\mathcal{P}\Phi$ is well-behaved away from the origin.

\begin{theorem}[Stability under white noise] We have
\begin{align*}
 (\mathcal{P}\Phi)(z) &=   \mathcal{N} \left(0,   \frac{1}{2\pi}  + \frac{1}{\pi}\frac{|z|^2}{1-|z|^2} \right) ~~\quad \mbox{for}~z \in \mathbb{D} \\
 (\mathcal{P}\Phi)(z) - (\mathcal{P}\Phi)(0)&=   \mathcal{N} \left(0,  \frac{1}{\pi}\frac{|z|^2}{1-|z|^2} \right) \qquad  \qquad \mbox{for}~z \in \mathbb{D} \\
\end{align*}
\end{theorem}
These properties guarantee that $\mathcal{P}\Phi$ is well-behaved in the interior. We see that, typically, $|(\mathcal{P}\Phi)(z) - (\mathcal{P}\Phi)(0)| \sim |z|$ for small values of $z$.

\subsection{A short glimpse at exponential convergence.}  As has already been pointed out by Nahon \cite{Nahon:2000Thesis}, at least in generic situations the unwinding series seems
to converge exponentially. Existing results \cite{Coifman_Steinerberger:2015,Qian:2010} are very far from showing that. The purpose of this section is to give a heuristic description 
of what could possibly be the dominant underlying dynamics.
\begin{thm}[Special case of Carleson's formula] Let $F:\mathbb{C} \rightarrow \mathbb{C}$ be holomorphic with roots $\left\{\alpha_i:i\in I\right\}$ (where $I$ is some index set) inside $\mathbb{D}$ and Blaschke factorization $ F = B \cdot G$. Then
$$ \int_{\mathbb{D}}{|F'(z)|^2dz} =  \int_{\mathbb{D}}{|G'(z)|^2dz}+ \frac{1}{2}\int_{\partial \mathbb{D}}{|G|^2 \sum_{i \in I}{\frac{1-|a_i|^2}{|z-\alpha_i|^2}} d\sigma},$$
where d$\sigma$ denotes the arclength measure on the boundary of the unit disk $\partial \mathbb{D}$.
\end{thm}
\noindent This immediately implies that the norm of consecutive elements in the Dirichlet space $\left(\| G_n\|_{\mathbb{D}}\right)_{n=1}^{\infty}$ is monotonically decreasing (see \cite{Coifman_Steinerberger:2015} for
a more extensive analysis), however, it is not possible to recover quantitative estimates because the quantities scale differently. Consider, for example, the function $F(z) = 2z + z^n$ for $n \geq 2$. Clearly, the only
root inside the unit disk is in $0$ and therefore
\begin{align*}
 \int_{\mathbb{D}}{|F'(z)|^2dz} &= \int_{\mathbb{D}}{4+n^2 |z|^{2n-2} dz} =  (n+4)\pi \\
\frac{1}{2}\int_{\partial \mathbb{D}}{|G|^2 \sum_{i \in I}{\frac{1-|a_i|^2}{|z-\alpha_i|^2}} d\sigma} &= \frac{1}{2}\int_{0}^{2\pi}{|F(e^{it})|^2 dt}  = 5\pi.
\end{align*}
We see that for $n$ large, the actual decrease does not correspond to a fixed proportion of the size but can be an arbitrarily large factor smaller than that. Exponential
convergence, however, is still be possible if cases like that do not occur often and if, whenever they occur, they do not occur for a large number of consecutive
iterations.

\begin{center}
\begin{figure}[h!]
\begin{tikzpicture}
\draw [->, thick] (-2,0) -- (2,0);
\draw [->, thick] (0,-2) -- (0,2);
\draw [black, thick,  domain=-0:6.28, samples=500] plot ({cos(100*\x) +  cos(1000*\x)/2  },{sin(100*\x) + sin(1000*\x)/2});

\draw [->, thick] (6-2,0) -- (6+2,0);
\draw [->, thick] (6,-2) -- (6,2);
\draw [black, thick,  domain=-0:6.28, samples=500] plot ({6+0.5*cos(100*\x) + cos(1000*\x) + 0.2*cos(100*\x)} ,{0.8*sin(100*\x) + 0.6*sin(1000*\x) + 0.4*sin(500*\x)  });
\end{tikzpicture}
\caption{$F(\partial \mathbb{D})$ for $f(z) = 2z + z^{10}$ compared to a more 'generic' function.}
\end{figure}
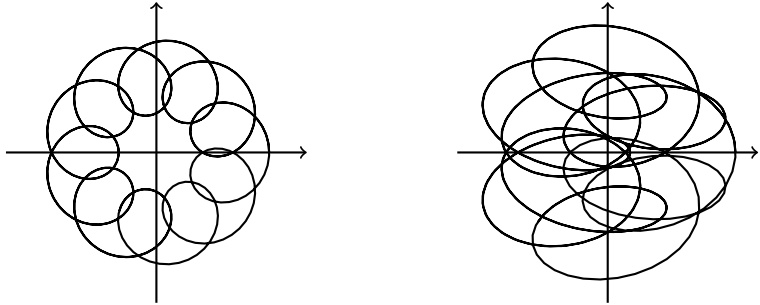
\end{center}

However, it is classical (and sometimes called the 'area-theorem') that the Dirichlet integral measures the area enclosed by $F(\partial \mathbb{D})$ weighted
with the winding number
$$ \int_{\mathbb{D}}{|F'(z)|^2dz}= \int_{\mathbb{C}}{\wind_{F(\partial \mathbb{D})}(z)dz}.$$
Moreover, assuming at least half of the roots of $F$ inside $\mathbb{D}$ to be at least distance $1/100$ from the boundary, we can estimate the weight
by using the argument principle to obtain
$$  \sum_{i \in I}{\frac{1-|a_i|^2}{|z-\alpha_i|^2}}  \gtrsim  \sum_{i \in I}{1} =  \wind_{F(\partial \mathbb{D})}(0).$$
If the winding number around 0 is comparable to the winding number at other points inside the unit disk (this is clearly violated in the case of $F(z) = 2z + z^n$ discussed above) and if $|G|$ does not vary too wildly, then from
elementary geometric considerations
$$ \int_{\partial \mathbb{D}}{|G|^2 \sum_{i \in I}{\frac{1-|a_i|^2}{|z-\alpha_i|^2}} d\sigma}  \gtrsim   \wind_{F(\partial \mathbb{D})}(0) \int_{\partial \mathbb{D}}{|G|^2  d\sigma} \sim  \int_{\mathbb{C}}{\wind_{F(\partial \mathbb{D})}(z)dz}.  $$
Summarizing, assuming the roots not to be too clustered close to the boundary of the unit disk and assuming the winding number around 0 to be not too small compared with the winding
number at other points inside $\mathbb{D}$ and assuming some moderate regularity on $|G|$, the weight gains half a derivative regularity and ensures that the decrease of the Dirichlet norm is proportional to the size. Numerical simulations suggest that violating any of these conditions, as in the case of the example $F(z) = 2z + z^n$, is not stable under Blaschke factorization. While we believe this to be the
underlying dominating dynamics, a systematic and rigorous justification is lacking.

\subsection{Explicit solvability.} There exists a small class of functions for which the unwinding series can be explicitly computed and coincides with the standard Fourier series. The argument
is completely elementary and, unfortunately, does seem to be way too specialized to give insight into the actual nonlinear dynamics at work. 

\begin{proposition} Let $0 \leq n_0 < n_1 < n_2 < \dots$ be a strictly increasing sequence of integers and
$$ f(z) = \sum_{k=0}^{\infty}{a_k z^{n_k}} \qquad \mbox{where, for all $n$,} \qquad |a_n| > \sum_{k=n+1}^{\infty}{|a_k|}.$$
Then the $N-$th term of the unwinding series is given by
$$ f(0) + a_1 B_1 + a_2 B_1 B_2 + \dots + a_N B_1 \cdots B_N = \sum_{k=0}^{N}{a_k z^{n_k}}.$$
\end{proposition}
\begin{proof} The proof is by induction. Clearly, the first term is $f(0)$. It suffices to remark that all arising roots are always in 0 because
for all $|z| \leq 1$
\begin{align*}
  \left| \sum_{k=N+1}^{\infty}{a_k z^{n_k}} \right| =  \left| a_{N+1}z^{n_{N+1}} + \sum_{k=N+2}^{\infty}{a_k z^{n_k}} \right| = |z|^{n_{N+1}} \left| a_{N+1} + \sum_{k=N+2}^{\infty}{a_k z^{n_k-n_{N+1}}} \right| 
\end{align*}
and we can estimate
$$ \left| a_{N+1} + \sum_{k=N+2}^{\infty}{a_k z^{n_k-n_{N+1}}} \right| \geq \left| a_{N+1} \right|  - \left| \sum_{k=N+2}^{\infty}{a_k z^{n_k-n_{N+1}}} \right| \geq \left| a_{N+1} \right| - \sum_{k=N+2}^{\infty}{|a_k|} > 0.$$
\end{proof}

We note that the condition on the coefficients can be iterated, which gives
$$  |a_n| > \sum_{k=n+1}^{\infty}{|a_k|} >  2\sum_{k=n+2}^{\infty}{|a_k|} > \dots > 2^{m-1} \sum_{k=n+m}^{\infty}{|a_k|} \geq  2^{m-1} |a_{n+m}|,$$
which immediately implies exponential decay of the sequence $(|a_n|)_{n=1}^{\infty}$.

\subsection{Poisson convolution and root detection.}\label{Section:RootDetection} 

One interesting application is the possible reconstruction of the location of various roots. Given a holomorphic function restricted to the boundary, the
interior values are uniquely defined by the Poisson integral
$$ u(r e^{i \theta}) = \frac{1}{2\pi} \int_{-\pi}^{\pi}{\frac{1-r^2}{1-2r\cos{(\theta-t)} + r^2} f(e^{it}) dt}.$$ 
In particular, for every fixed $0 \leq r < 1$, the map to a smaller disk of radius $r$
$$ F \big|_{\partial \mathbb{D}} \rightarrow u(r e^{it})$$
is explicitly given as the convolution with the Poisson kernel of fixed with and easy and fast to compute. Assuming there is no root at distance $r$ from the origin, this gives a new function $u_{r}:\partial\mathbb{D} \rightarrow \mathbb{C}$ and its
Blaschke decomposition 
$$ u_r = B_r G_r \qquad \mbox{where} \quad B_r = z^m\prod_{|\alpha_i| < r}{  \frac{z-\frac{\alpha_i}{r}}{1-\frac{\overline{\alpha_i}}{r} z}}$$
ranges over all roots $|\alpha_i| < r$. Note that its instantaneous frequency is given by
$$ \phi_r'(t) = m + \sum_{|\alpha_i| < r}{ \frac{1-|\alpha_i/r|^2}{|e^{i \theta} - \alpha_i/r|^2}},$$
which is merely the sum over Poisson kernels at rescaled roots: this means that by constructing the Blaschke product, we are able to get a rough impression where
the roots of the function are located. We refer to Section \ref{Section:Numerics:RootDetection} for numerical examples.

\section{Proofs}
\subsection{Proof of Theorem 1.} We wish to estimate the $L^2-$norm contained in the negative Fourier frequencies of $f(t) = A(t)e^{i \phi(t)}$. The crucial insight is that the assumptions
on the function imply decay of the oscillatory integral by the classical van der Corput estimate. 
\begin{proof}
Using Plancherel's theorem gives
$$ \| P_{-} f\|_{L^2(\mathbb{T})}^2 = \sum_{n = 1}^{\infty}{ \left| \int_{0}^{2\pi}{A(t) e^{i \phi(t) + i n t}dt} \right|^2}.$$
We will now estimate this sum term-by-term. Since $\phi' > 0$, we may use integration by parts to get
\begin{align*}
\int_{0}^{2\pi}{A(t) e^{i \phi(t) + i n t}dt} &= \int_{0}^{2\pi}{ \frac{A(t)}{i \phi'(t) + in} \frac{d}{dt} e^{i \phi(t) + i n t}dt} \\
&= -  \int_{0}^{2\pi}{ \left( \frac{A'(t)}{i \phi'(t) + in} - \frac{2 A(t) \phi''(t) }{(i \phi'(t) + i n)^2} \right)  e^{i \phi(t) + i n t}dt}.
\end{align*}
Taking absolute values yields
\begin{align*}
\left| \int_{0}^{2\pi}{A(t) e^{i \phi(t) + i n t}dt} \right| &\leq   \int_{0}^{2\pi}{ \left| \frac{A'(t)}{i \phi'(t) + in} \right| + \left| \frac{2 A(t) \phi''(t) }{(i \phi'(t) + i n)^2} \right| dt}.
\end{align*}
The first term can be easily bounded as
$$ \int_{0}^{2\pi}{ \left| \frac{A'(t)}{i \phi'(t) + in} \right|dt} \leq  \int_{0}^{2\pi}{ \frac{A'(t)}{| \phi'(t) + n|} dt} \leq 2 \pi \sup_{0 \leq t \leq 2\pi}{\frac{A'(t)}{\phi'(t) + n}}.$$
For the second term we use $|\phi''(t)| \leq \varepsilon \phi'(t)$
$$  \int_{0}^{2\pi}{ \left| \frac{2 A(t) \phi''(t) }{(i \phi'(t) + i n)^2} \right| dt} \leq \int_{0}^{2\pi}{ \frac{2 A(t) \varepsilon \phi'(t) }{( \phi'(t) + n)^2} dt} \leq 2\pi \varepsilon \|A\|_{L^{\infty}}  \sup_{0 \leq t \leq 2\pi}{ \frac{\phi'(t)}{( \phi'(t) + n)^2} }.$$
The Cauchy-Schwarz inequality in its most elementary form $(a+b)^2 \leq 2(a^2 + b^2)$ implies
\begin{align*}
 \| P_{-} f\|_{L^2(\mathbb{T})}^2 &= \sum_{n = 1}^{\infty}{ \left| \int_{0}^{2\pi}{A(t) e^{i \phi(t) + i n t}dt} \right|^2} \\
&\leq \sum_{n = 1}^{\infty}{\left(    2 \pi  \sup_{0 \leq t \leq 2\pi}{\frac{A'(t)}{\phi'(t) + n}}  +  2\pi \varepsilon \|A\|_{L^{\infty}} \sup_{0 \leq t \leq 2\pi}{ \frac{\phi'(t)}{( \phi'(t) + n)^2}}  \right)^2 }  \\
&\leq 2\sum_{n = 1}^{\infty}{  \left( 2 \pi  \sup_{0 \leq t \leq 2\pi}{\frac{A'(t)}{\phi'(t) + n}} \right)^2 } + 2\sum_{n = 1}^{\infty}{\left(  2\pi \varepsilon \|A\|_{L^{\infty}} \sup_{0 \leq t \leq 2\pi}{ \frac{\phi'(t)}{( \phi'(t) + n)^2}}  \right)^2}
\end{align*}
The first term can be bounded from above by
\begin{align*}
2\sum_{n = 1}^{\infty}{  \left( 2 \pi  \sup_{0 \leq t \leq 2\pi}{\frac{A'(t)}{\phi'(t) + n}} \right)^2 }  &\leq 8 \pi^2 \|A'\|_{L^{\infty}}^2 \sum_{n = 1}^{\infty}{  \left( \frac{1}{\inf_{0 \leq t \leq 2\pi}{\phi'(t)} + n} \right)^2 } 
\end{align*}
and this sum is easily seen to be dominated by
$$ \sum_{n = 1}^{\infty}{  \left( \frac{1}{\inf_{0 \leq t \leq 2\pi}{\phi'(t)} + n} \right)^2 }  \leq \int_{\inf_{0 \leq t \leq 2\pi}{\phi'(t)}}^{\infty}{\frac{1}{z^2}dz} = \frac{1}{\inf_{0 \leq t \leq 2\pi}{\phi'(t)}}.$$
The second sum can be dealt with analogously
\begin{align*}
 2\sum_{n = 1}^{\infty}{\left(  2\pi \varepsilon \|A\|_{L^{\infty}} \sup_{0 \leq t \leq 2\pi}{ \frac{\phi'(t)}{( \phi'(t) + n)^2}}  \right)^2} &\leq 8 \pi^2 \varepsilon^2  \|A\|_{L^{\infty}}^2 \sum_{n=1}^{\infty}{ \left(\sup_{0 \leq t \leq 2\pi}{\frac{1}{( \phi'(t) + n)} }\right)^2}  \\
&\leq  \frac{8 \pi^2 \varepsilon^2}{\inf_{0 \leq t \leq 2\pi}{\phi'(t)}}.
\end{align*}
Finally, for comparison, we note that trivially
$$ \| f\|_{L^2}^2 =   \int_{0}^{2\pi}{|A(t) e^{i \phi(t)}|^2dt} =   \int_{0}^{2\pi}{|A(t)|^2dt} = \|A\|^2_{L^2}.$$
\end{proof}

\subsection{Proof of Theorem 2.}
\begin{proof} The proof uses an elementary geometric consideration. Fix $0 \leq t \leq 2\pi$ and suppose the phases $\phi(t)$ and $\phi^*(t)$ of
$$ A(t) e^{i \phi(t)} \qquad \mbox{and} \qquad P_{+}(A(t) e^{i \phi(t)} ) = |P_{+}(A(t) e^{i \phi(t)} )| e^{i \phi^*(t)}$$
differ by some angle $\alpha$. Then
$$ | A(t) e^{i \phi(t)} - P_{+}(A(t) e^{i \phi(t)} )| \qquad \mbox{cannot be arbitrarily small.}$$
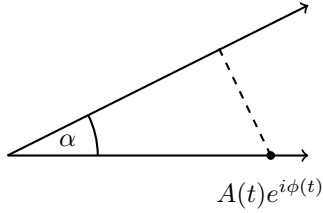
\begin{figure}[h!]
\begin{tikzpicture}
\draw [thick, ->] (0,0) -- (4,0);
\filldraw (3.5,0) circle (0.05cm);
\node at (3.5,-0.5) {$A(t) e^{i \phi(t)}$};
\draw [thick, ->] (0,0) -- (4,2);
\draw [thick] (1.2,0) arc (0:27:1.2cm);
\node at (0.8,0.2) {$\alpha$};
\draw [dashed, thick] (3.5,0) -- (2.8,1.45);
\end{tikzpicture}
\caption{Projection onto another phase.}
\end{figure}

Simple geometric considerations show that we have
$$  | A(t) e^{i \phi(t)} - P_{+}(A(t) e^{i \phi(t)} )| \geq \begin{cases} 
A(t)\sin{\alpha} \qquad &\mbox{if}~ -\frac{\pi}{2} \leq \alpha \leq \frac{\pi}{2}\\
A(t) \qquad &\mbox{otherwise.}
\end{cases}$$
Note that, by convention, the distance between the phases on the torus $|\phi(t) - \phi^*(t)|_{\mathbb{T}}$ is at most $\pi$ and thus
using the concavity of $\sin{x}$ on that interval, we get that
\begin{align*}
  | A(t) e^{i \phi(t)} - P_{+}(A(t) e^{i \phi(t)} )| &\geq   \begin{cases} 
A(t) \sin{\left( |\phi(t) - \phi^*(t)|_{\mathbb{T}} \right)} \qquad &\mbox{if}~ 0 \leq \alpha \leq \frac{\pi}{2}\\
A(t) \qquad &\mbox{if}~ \frac{\pi}{2} \leq \alpha \leq \pi \end{cases} \\
&\geq A(t) \frac{ |\phi(t) - \phi^*(t)|_{\mathbb{T}}}{\pi} \\
&\geq \left( \inf_{0 < t < 2\pi}{A(t)} \right) \frac{ |\phi(t) - \phi^*(t)|_{\mathbb{T}}}{\pi}.
\end{align*}
Rearranging, taking squares and integrating over $\mathbb{T}$ implies that
$$ \| \phi(t) - \phi^*(t)\|_{L^2(\mathbb{T})}^2  \leq \frac{\pi^2}{  \inf_{0 < t < 2\pi}{A(t)^2}} \|  A(t) e^{i \phi(t)} - P_{+}(A(t) e^{i \phi(t)} ) \|_{L^2(\mathbb{T})^2}^2.$$
However, since $P_{+}$ is an orthogonal projection and $P_{+} \oplus P_{-} = \mbox{id}$, we get from the Phytagorean theorem that
$$\|  A(t) e^{i \phi(t)} - P_{+}(A(t) e^{i \phi(t)} ) \|_{L^2(\mathbb{T})^2}^2 = \| P_{-}(A(t) e^{i \phi(t)} )\|_{L^2(\mathbb{T})}^2.$$
Our main result now implies the desired statement.
\end{proof}

\subsection{Proof of Theorem 3}
\begin{proof} Our proof consists of a detailed analysis of the Poisson extension of white noise $\mathcal{P}\Phi$. We fix a $\varepsilon > 0$ arbitrary and
only consider $\mathcal{P}\Phi$ in the smaller disk $B(0, 1-\varepsilon)$. We will proceed by breaking up the boundary into $\sim \varepsilon^{-1}$ 
intervals of length $\varepsilon$, performing an analysis and then sending $\varepsilon$ to 0. There are various ways of introducing the precise structure of white noise $\Phi$, however, our argument actually only requires that (1) it exists as a stochastic process. that (2) for all intervals $[a,b] \subset \mathbb{T} \cong [0, 2\pi]$
$$ \int_{a}^{b}{\Phi(t) dt} = \frac{1}{2\pi}\mathcal{N}(0, b-a)$$
and that (3) for disjoint intervals the arising two random variables are independent.
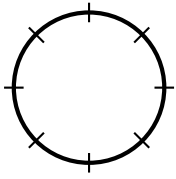
\begin{figure}[h!]
\begin{tikzpicture}
\draw [ultra thick] (0,0) circle (1cm);
\draw [thick] (0.87,0) -- (1.13,0);
\draw [thick] (-0.87,0) -- (-1.13,0);
\draw [thick] (0,0.87) -- (0,1.13);
\draw [thick] (0,-0.87) -- (0,-1.13);
\draw [thick] (0.6,0.6) -- (0.8,0.8);
\draw [thick] (-0.6,0.6) -- (-0.8,0.8);
\draw [thick] (0.6,-0.6) -- (0.8,-0.8);
\draw [thick] (-0.6,-0.6) -- (-0.8,-0.8);
\end{tikzpicture}
\caption{Breaking up the boundary into $\sim \varepsilon^{-1}$ intervals of length $\varepsilon$.}
\end{figure}

The Poisson kernel is explicitly given by
$$ P_r(\theta) = \frac{1}{2\pi} \frac{1-r^2}{1 - 2r \cos{\theta} + r^2}$$
we can easily get the correct asymptotics and can deduce that the measure induced on the boundary by the Poisson kernel associated to a point $z \in \mathbb{D}$ has most of its support
on an interval of length $\sim 1-|z|.$ We recall the addition law for independent Gaussian variables
$$ a\mathcal{N}(\mu_1, \sigma_1) +  b\mathcal{N}(\mu_2, \sigma_2)  \sim  \mathcal{N}(a \mu_1 + b\mu_2, a^2\sigma_1 + b^2\sigma_2)$$  
from which it follows, by taking $\varepsilon$ sufficiently small, that
$$  (\mathcal{P}\Phi)(z) = \mathcal{N}\left( 0, \int_{0}^{2\pi}{P_r(\theta)^2 d\theta}\right).$$
This requires us to determine the $L^2-$norm of the Poisson kernel. We use the representation
$$ P_r(\theta) = \frac{1}{2\pi} \frac{1-r^2}{1 - 2r \cos{\theta} + r^2} = \frac{1}{2\pi}\left(1 + \sum_{n \neq 0}{|r|^n e^{in \theta}}\right)$$
to compute 
\begin{align*}
 \int_{0}^{2\pi}{  \frac{1}{4\pi^2}\left(1 + \sum_{n \neq 0}{|r|^n e^{in \theta}}\right)^2 d\theta}&= \frac{1}{4\pi^2}\int_{0}^{2\pi}{
\left(1 + \sum_{n \neq 0}{|r|^n e^{in \theta}}\right)^2 d\theta} \\
&= \frac{1}{2\pi} + \frac{1}{4\pi^2}\int_{0}^{2\pi}{ \left( \sum_{n \neq 0}{|r|^n e^{in \theta}}\right)^2 d\theta} \\
&=  \frac{1}{2\pi}  + \frac{1}{2\pi}{\sum_{n \neq 0}{|r|^{2n}}} \\
&=  \frac{1}{2\pi}  + \frac{1}{\pi}\frac{r^2}{1-r^2}.
\end{align*}
The second statement can be proven in a similar way. Note that
$$\mathcal{P}\Phi(z) - \mathcal{P}\Phi(0) = \mathcal{N}\left( 0, \int_{0}^{2\pi}{\left(P_r(\theta)-\frac{1}{2\pi}\right)^2 d\theta}\right).$$
The same computation as before now gives
\begin{align*}
\int_{0}^{2\pi}{ \left(P_r(\theta)-\frac{1}{2\pi}\right)^2  d\theta} &= \int_{0}^{2\pi}{         \left(\frac{1}{2\pi}    \sum_{n \neq 0}{|r|^n e^{in \theta}}    \right)^2          d\theta} \\
&= \frac{1}{4\pi^2}\int_{0}^{2\pi}{\left(\sum_{n \neq 0}{|r|^n e^{in \theta}}\right)^2 d\theta} \\
&= \frac{1}{\pi} \frac{r^2}{1-r^2}.
\end{align*}
\end{proof}

The final part of the previous statement serves as an explanation for the stability of the Blaschke product under perturbation by
additive white noise since
$$ (F + \mathcal{P}\Phi)(z) - (F + \mathcal{P}\Phi)(0) = (F(z) - F(0)) + (\mathcal{P}\Phi)(z) - (\mathcal{P}\Phi)(0)).$$
The function $F(z) - F(0)$ has roots that are now perturbed by adding another function. However, since 
$$ \mathbb{E}\left|(\mathcal{P}\Phi)(z) - \mathcal{P}\Phi)(0))\right| \sim |z| \qquad \mbox{for}~ \left|z \right| \leq \frac{3}{4},$$
we see that the function being added is very small in any small neighborhood of the origin. The precise effect

\section{Numerical examples}

We show numerical results of the proposed algorithm. The Matlab code and simulated data are available upon request. 
Given a signal $f$ defined on $\mathbb{T}$ and denote $F:=P_+f$. We use the following notations for the Blaschke decomposition algorithm. Denote $G_0:=F$. \begin{enumerate}
\item Fix $i=0$.
\item Decompose $G_i=L_{i+1}+H_{i+1}$, where $L_{i+1}:=\mathfrak{F}_DG_i$ and $H_{i+1}:=(1-\mathfrak{F}_D)G_i$, $\mathfrak{F}_D$ means a chosen low pass filter by the $(D-1)$-th order polynomial, and $D\in \mathbb{N}$ is determined by the user. 
\item Apply the Blaschke decomposition on $H_{i+1}$ and get $H_{i+1}=B_{i+1}G_{i+1}$.
\item Set $i\leftarrow i+1$ and iterate (2)-(3) for $K$ times, where $K\in\mathbb{N}$ is determined by the user.  
\end{enumerate}
We view $L_1$ as the ``local DC'' term of the signal $f$, 
\[
\tilde{f}_l:=L_{l+1}\prod_{k=1}^lB_k 
\]
as the $l$-th decomposed oscillatory component with the amplitude modulation (AM) $L_{l+1}$ with the instantaneous frequency (IF) determined by the derivative of the phase of $\prod_{k=1}^lB_k$, where $l=1,2,\ldots$. 
Note that when $D=1$, the effect of low pass filter is removing the mean of $G_i$, while when $D>1$, a $(D-1)$-order polynomial is fitted to $G_i$. Clearly, when $i>1$, the removed mean or $(D-1)$-order polynomial becomes the amplitude of $\tilde{f}_{i-1}$. 
We apply the Guido and Mary Weiss algorithm \cite{Weiss_Weiss:1962} to estimate the Blaschke decomposition of a given function $F$. When we evaluate the Blaschke decomposition of $F$ as $F=BG$ by $G:=e^{P_+(\ln|F|)}$, we might encounter $0$ inside the log function. To stabilize the numerical evaluation, we evaluate $G$ by 
\[
G=e^{P_+(\ln\sqrt{|F|^2+\epsilon^2})}, 
\]
where $\epsilon>0$ is chosen by the user.
There are several different ways of estimating the AM and IF of the decomposed component, like SST or the phase gradient estimation \cite{Nahon:2000Thesis}. 

\begin{figure}[ht]
\includegraphics[width=.95\textwidth]{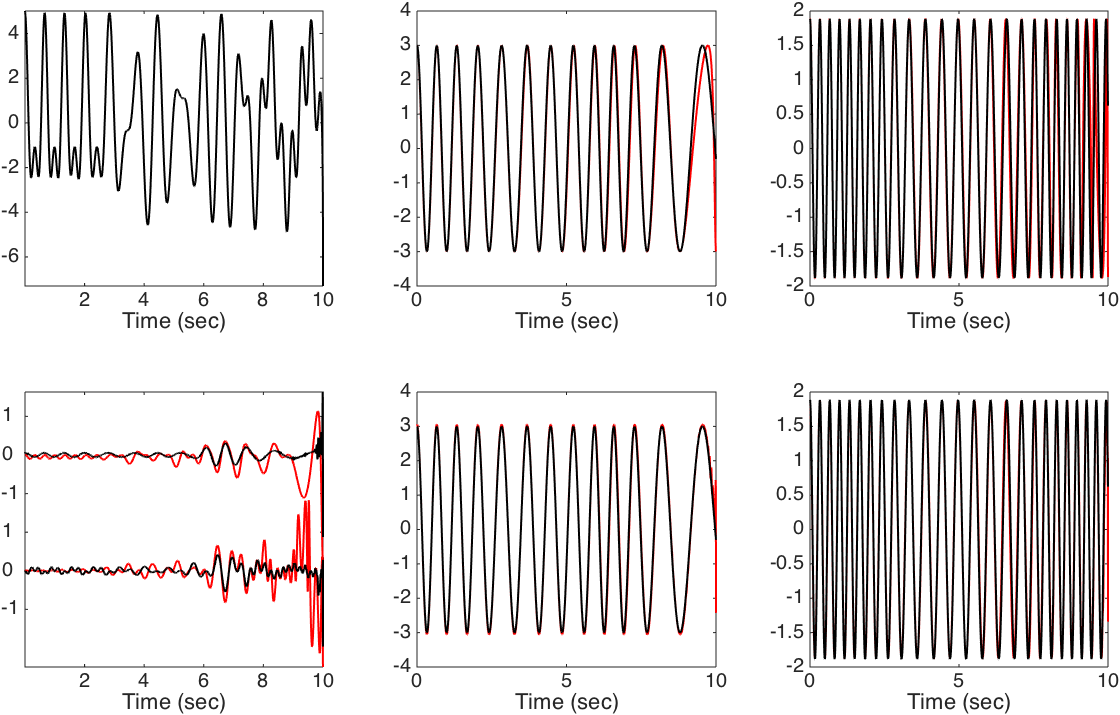}
\captionsetup{width=0.95\textwidth}
\caption{An illustration of the Blaschke decomposition of $f=f_1+f_2$, where $f_1$ and $f_2$ have close instantaneous frequencies. Top left: $f$; top middle: $f_1$ is shown in black and $\tilde{f}_1$ is shown in red; top right: $f_2$ is shown in black and $\tilde{f}_2$ is shown in red; bottom left: the upper red curve is $f_1-\tilde{f}_1$, the upper black curve is $f_1-\tilde{f}^c_1$, the lower red curve is $f_1-\tilde{f}_1$ and the lower black curve is $f_1-\tilde{f}^c_1$; bottom middle: $f_1$ is shown in black and $\tilde{f}^c_1$ is shown in red; bottom right: $f_2$ is shown in black and $\tilde{f}^c_2$ is shown in red. It is clear that the carrier frequency helps to increase the decomposition accuracy.}
\label{Example:1}
\end{figure}

Numerically, to avoid the boundary effect, we apply the following reflection trick in practice; that is, for the observed signal $f$ on time $[0,1]$, we analyze $\bar{f}$, which is defined on $[0,2]$ by
\[
\bar{f}|_{[0,1]}(t)=f(t);\,\bar{f}|_{(1,2]}(s)=f(2-s),
\]
and the final results come from restricting the decomposition on $[0,1]$. Note that in general, the signal might be recorded for a period $T>0$ longer than $1$ second with a sampling rate $K>0$ Hz. To analyze this kind of signal, we could rescale the signal to 1 second with the sampling rate $TK$, run the analysis, and scale back to the original length and sampling rate. Unless indicated differently, below we run the Blaschke decomposition with $D=1$ and $\epsilon=10^{-4}$.

We quickly summarize how we estimate IF by SST here. For a given function $f$ and a window function $h$ in the proper space, for example, $f$ is in the tempered distribution space, and $h$ is in the Schwartz space, STFT is defined as
\begin{equation}
V^{(h)}_f(t, \omega) = \int f(s) h(s-t)e^{-j2\pi\omega s} ds,
\label{eq: stft1}
\end{equation}
where $t\in\RR$ is the time and $\omega\geq 0$ is the frequency. It is well known that STFT is blurred due to the Heisenberg uncertainty principle. To sharpen the TFR determined by STFT, we could consider SST, which is a special reassignment technique \cite{Kodera_Gendrin_Villedary:1978,Auger_Flandrin:1995,Auger_Chassande-Mottin_Flandrin:2012}. SST counts on the {\em frequency reassignment rule} to sharpen the TFR, which is defined by:
\begin{align}
\Omega_f(t,\eta)=-\Im\frac{V_f^{(\mathcal{D}h)}(t,\eta)}{2\pi V_f^{(h)}(t,\eta)}\mbox{ when }  |V_f^{(h)}(t,\eta)|>\Theta\label{RM:omega}
\end{align}
and $\Omega_f(t,\eta)=-\infty$ otherwise,
where $\Im$ means taking the imaginary part, $\Theta\geq0$ is the chosen hard threshold and $\mathcal{D}h:=h'$, the first derivative of $h$.

\begin{figure}[ht]
\includegraphics[width=.95\textwidth]{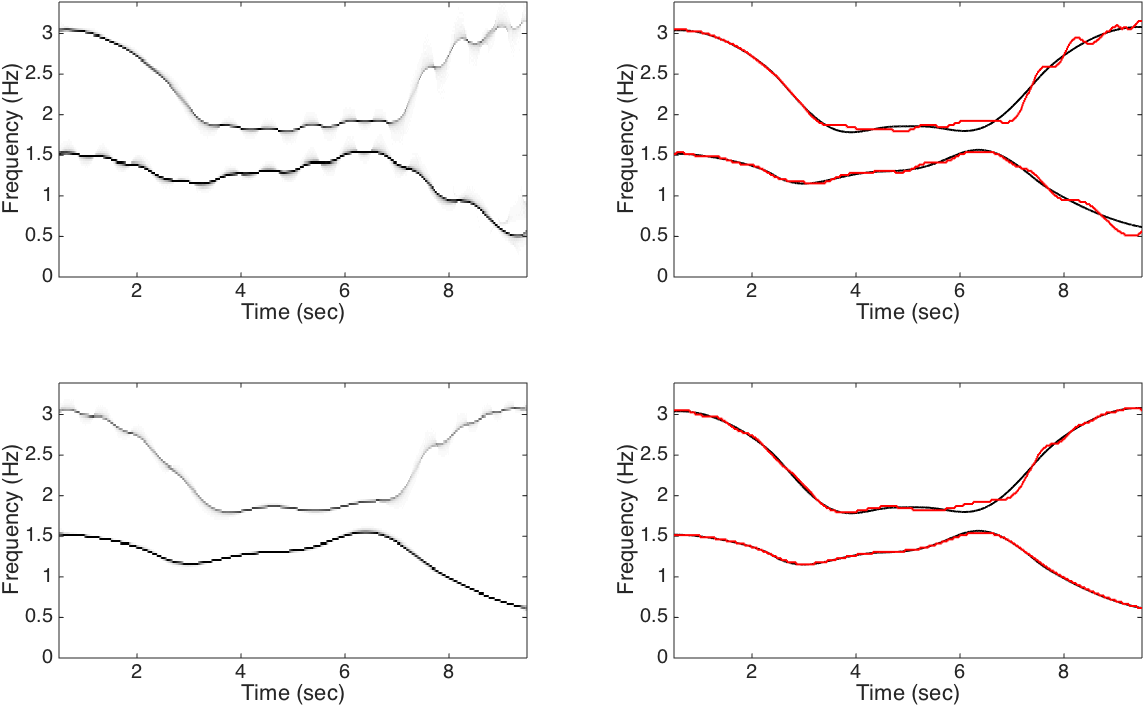}
\captionsetup{width=0.95\textwidth}
\caption{The Blaschke time-frequency (TF) representation $B_f$ is shown on the left upper subplot; the Blaschke TF representation $B_f^c$ is shown on the left lower subplot; right upper subplot: the true instantaneous frequency (IF) of $f_1$ and $f_2$ are shown as the black curves, and the estimated IF's from $f$ are shown as the red curves;  right lower subplot: the true IF's of $f_1$ and $f_2$ are shown as the black curves, and the estimated IF's from $f^c$ are shown as the red curves. We could see that the IF estimated from $f^c$ is more accurate.}
\label{Example:2}
\end{figure}

SST is defined by nonlinearly moving STFT coefficients {\em only} on the frequency axis guided by the frequency reassignment rule
\begin{align}\label{definition:SST}
S_f(t,\xi):=\int_{\{\eta:\,|V^{(h)}_f(t,\eta)|>\Theta\}}V^{(h)}_f(t,\eta)g_\alpha\left(|\xi+\eta-\Omega_f(t,\eta)|\right)\ud \eta,
\end{align}
where $t\in\RR$, $\xi\geq0$, $0<\alpha\ll 1$ is chosen by the user, $g_{\alpha}(\cdot):=\frac{1}{\alpha}g(\frac{\cdot}{\alpha})$ and $g$ is a smooth function so that $g_\alpha$ weakly converges to the Dirac measure as $\alpha\to 0$. The TF representation of an adaptive harmonic function determined by SST is highly concentrated on the location representing the IF, thus by any available curve extraction algorithm, we could accurately estimate the IF. To estimate the IF, we apply \cite[(15)]{Chen_Cheng_Wu:2014} with the penalty term for the regularity of IF weighted by $\lambda\geq 0$. The main reason we apply SST to estimate IF from a given oscillatory component is its robustness to different kinds of noise, even when the noise is non-stationary. This property has been extensively studied in \cite{Chen_Cheng_Wu:2014}. For the background, theoretical analysis and algorithmic details of SST, we refer the reader to, for example \cite{Daubechies_Lu_Wu:2011,Daubechies_Wang_Wu:2016}. Last but not least, we introduce the Blaschke TF representation. Given a $L^2$ function $f$ and its Blaschke decomposition $f=\sum_{k=1}^K\tilde{f}_k+R$, where $R$ is the remainder term, we define a new TF representation of $f$, called the \textit{Blaschke TF representation}, denoted as 
\[
B_f:=\sqrt{\sum_{k=1}^K|S_{f_k}|^2}. 
\]
Note that in general $B_f\neq |S_f|$ since SST is a nonlinear operator.
We mention that we could also consider RM or other nonlinear TF analysis techniques to estimate the IF, but in this paper we focus on SST.

\subsection{Two components with close IF's}

The first example shows that the Blaschke decomposition works well for the adaptive harmonic model.
Take $W$ to be the standard Wiener process defined on $\RR$ and define a {\it smoothed Wiener process with bandwidth $\sigma>0$} as
\begin{align}
\Phi_{\sigma}:=W\star g_{\sigma},
\end{align}
where $g_{\sigma}$ is the Gaussian function with the standard deviation $\sigma>0$ and $\star$ denotes the convolution operator. 
Take $L>0$, $\xi_0>0$ and $c_0\geq0$. Define the following random process $\phi^{(\xi_0,c_0)}$ on $[0,L]$ by
\begin{align}\label{Simulation:Phi}
\phi^{(\xi_0,c_0)}(t)=\xi_0 t+c_0\int_{0}^t \frac{\Phi_{\sigma_\phi}(s)}{\|\Phi_{\sigma_\phi}\|_{L^\infty[0,L]}}\ud s,\nonumber
\end{align}
where $t\in[0,L]$ and $\sigma_\phi>0$.
Note that $\phi^{(\xi_0,c_0)}$ is a monotonically increasing random process and in general there is no close form expression of $\phi^{(\xi_0,c_0)}$.

\begin{figure}[ht]
\includegraphics[width=.95\textwidth]{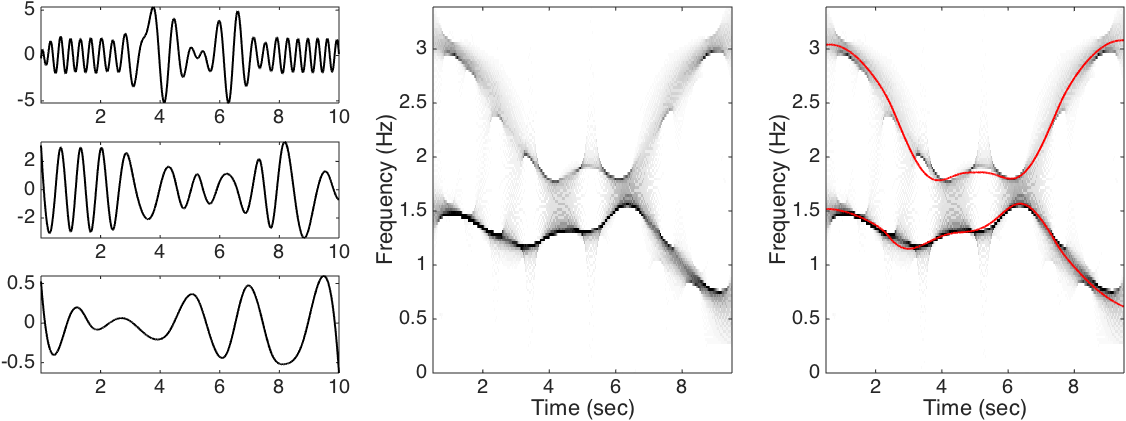}
\captionsetup{width=0.95\textwidth}
\caption{Left three subplots show the first three intrinsic mode functions (IMF) determined by the empirical mode decomposition (EMD). The time-frequency (TF) representation of $f$ determined by SST, denoted as $S_f$ is shown on the middle subplot. The right subplot shows $S_f$ superimposed with the truth instantaneous frequencies. Clearly, the first three IMFs determined by EMD are far from the ground truth. For SST, although we could obtain information about both components from $S_f$, when compared with $B_f$ shown in Figure \ref{Example:2}, there are interference pattern in $S_f$, in particular during the period $[4,7]$ second.}
\label{Example:3}
\end{figure}

We then generate two oscillatory components as
\begin{equation}
f_1(t)=e^{i2\pi \phi_1(t)},\quad
f_2(t)=e^{i2\pi \phi_2(t)},\quad
f(t)=f_1(t)+f_2(t),\label{Equation:Definiteion:Simulation:Clean}
\end{equation}
where $\phi_1(t)$ is sampled from $\phi^{(\pi/2,1)}$ and $\phi_2(t)$ is independently sampled from $\phi^{(3,1)}$. 
The goal is to decompose $f_1(t)$ and $f_2(t)$ from $f$ and estimate the IF's $\phi'_1(t)$ and $\phi'_2(t)$. We choose both components to be of constant amplitudes to show the results to simplify the discussion.
Denote $\tilde{f}_i(t)$, $i=1,2$ as the $i$-th decomposed components of $f(t)$ by the Blaschke decomposition, and denote $\tilde{\phi}'_i(t)$ as the estimated IF of $\phi'_i(t)$. 
To show the implication of Theorem \ref{Theorem:CarrierPhase} that the carrier frequency helps the estimation accuracy of the Blaschke decomposition, we consider $f^c(t):=f(t)e^{i2\pi \xi_0t}$, where $f^c$ is $f$ with a $\xi_0$ Hz carrier frequency and $\xi_0\geq 0$. Denote $\tilde{f}^{c0}_i(t)$ as the $i$-th decomposed components of $f^c$ by the Blaschke decomposition. The decomposed components of $f$ by taking the carrier frequency into account are thus $\tilde{f}^c_i(t):=\tilde{f}^{c0}_i(t)e^{-i2\pi\xi_0 t}$. The estimated IF of $\phi'_i(t)$ from $\tilde{f}^c_i(t)$ is denoted as $\tilde{\phi}^{\prime c}_i(t)$. In this example, we choose $\xi_0=20$.  

Numerically, we take $\sigma_\phi=200$, $L=10$ and sample $f(t)$ at the sampling rate $512$ Hz. For SST, we set $\Theta=10^{-4}\%$ of the root mean square energy of the signal under analysis and $\alpha$ small enough so that $g_\alpha$ is implemented as a discretization of the Dirac measure. The frequency axis is uniformly discretized into $0.0128$Hz. When we estimate IF, we take $\lambda=1$. 
The decomposition results are shown in Figures \ref{Example:1}. Visually, the decomposition fits the ground truth well, and the carrier frequency helps to increase the decomposition accuracy.
To quantify how well the decomposition is, we consider the following measurement. For the estimator of the quantity $f_i(t)$, denoted as $\tilde{f}_i(t)$, where $i=1,2$, the error ratio (ER) is defined as
\[
\text{ER}(\tilde{f}_i)=\frac{\|\tilde{f}_i-f_i\|_{L^2}}{\|f_i\|_{L^2}}.
\]
We use the same quantity to evaluate the estimator of the IF. Then we have that $\text{ER}(\tilde{f}_1)=0.16$, $\text{ER}(\tilde{f}_2)=0.31$,
 $\text{ER}(\tilde{f}^c_1)=0.05$, and $\text{ER}(\tilde{f}^c_2)=0.1$. We could see the improved accuracy by the carrier frequency.
In Figure \ref{Example:2}, the SST is applied to $\tilde{f}_1$ and $\tilde{f}_2$ separately to obtain $S_{\tilde{f}_1}$ and $S_{\tilde{f}_2}$ and hence the IF's of $\tilde{f}_1$ and $\tilde{f}_2$. 
Similarly, we could apply the same procedure to $f^c$, and determine the IF's of $\tilde{f}^c_1$ and $\tilde{f}^c_2$, as well as the associated Blaschke TF representation, denoted as $B_f^c$.
We could see that each component is well approximated, and hence its IF, and the improvement is clear when we apply the 20Hz carrier frequency -- $\text{ER}(\tilde{\phi}'_1)=0.035$, $\text{ER}(\tilde{\phi}'_2)=0.031$,
 $\text{ER}(\tilde{\phi}^{\prime c}_1)=0.019$, and $\text{ER}(\tilde{\phi}^{\prime c}_2)=0.017$.
We mention that this is a challenging example since the IF's of the two components are very close during the period $[4,7]$. We mention that to decompose $f$ to $f_1$ and $f_2$ by the other time-frequency analysis techniques, a precise choice of the window is needed due to the close IF. For example, the well-known EMD fails and the decomposition is far deviated from the ground truth; the SST is impacted by the interference caused by the close IF's. See Figure \ref{Example:3} for an example. With the Blaschke decomposition, we could alleviate this limitation and get the precise information we have interest.

\begin{figure}[h!]
\includegraphics[width=.95\textwidth]{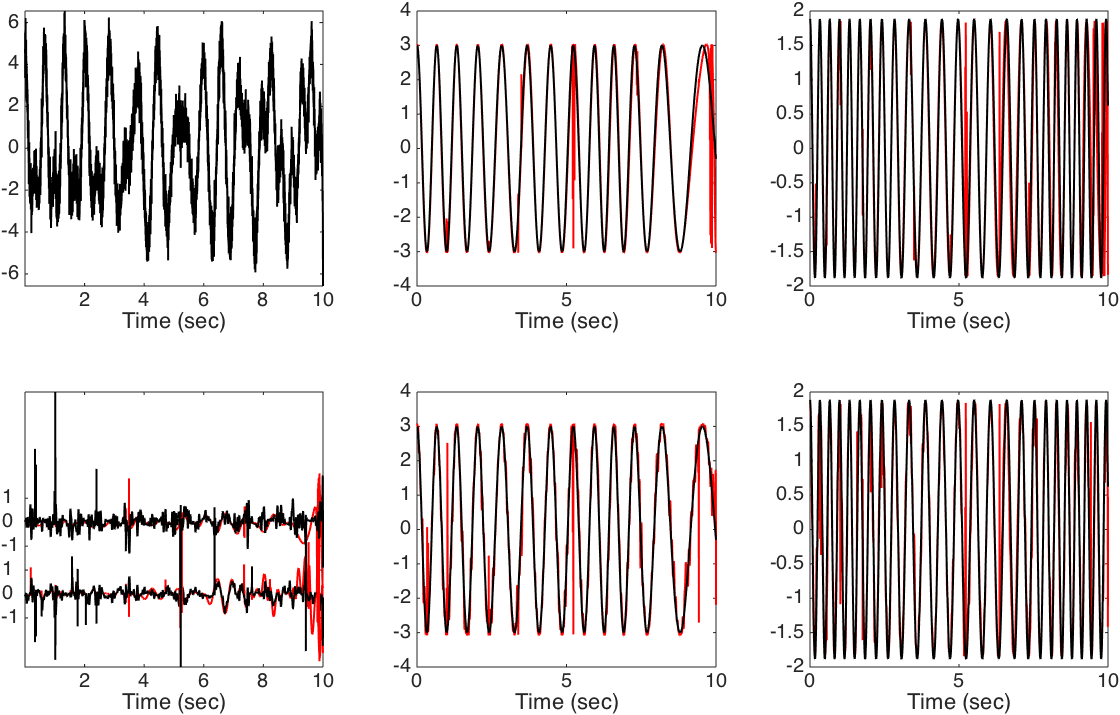}
\captionsetup{width=0.95\textwidth}
\caption{An illustration of the Blaschke decomposition of $Y=f+\sigma\xi$. Top left: $Y$; top middle: $f_1$ is shown in black and $\tilde{Y}_1$ is shown in red; top right: $f_2$ is shown in black and $\tilde{Y}_2$ is shown in red; bottom left: the upper red curve is $f_1-\tilde{f}_1$, the upper black curve is $f_1-\tilde{Y}^c_1$, the lower red curve is $f_1-\tilde{f}_1$ and the lower black curve is $f_1-\tilde{Y}^c_1$; bottom middle: $f_1$ is shown in black and $\tilde{Y}^c_1$ is shown in red; bottom right: $f_2$ is shown in black and $\tilde{Y}^c_2$ is shown in red. It is clear that the carrier frequency helps to increase the decomposition accuracy.}
\label{Example:4}
\end{figure}

\subsection{Stability under White Noise}

In this subsection, we show the stability of the Blaschke decomposition to different kinds of noises. Firstly, we consider the additive noise $Y(t)=f(t)+\sigma \xi(t)$, where $f$ is defined in (\ref{Equation:Definiteion:Simulation:Clean}), $\xi(t)$ is a Gaussian white noise with mean 0 and standard deviation 1, and $\sigma>0$ is chosen so that the signal to noise ratio (SNR) is 10, where SNR is defined as $20\log\frac{\text{std}(f)}{\sigma}$ and std means the standard deviation. Again, consider $Y^c(t)=e^{i2\pi\xi_0 t}Y(t)$ with $\xi_0=20$ Hz. 
Secondly, we consider the multiplicative noise $Z(t)=f(t)e^{\xi(t)/2}$, where $\xi$ is the same noise as that in $Y(t)$. Again, consider $Z^c(t)=e^{i2\pi\xi_0 t}Z(t)$. 

Denote $\tilde{Y}_i(t)$, $i=1,2$ as the $i$-th decomposed components of $Y(t)$ by the Blaschke decomposition. 
Similarly, denote $\tilde{Y}^{c0}_i(t)$ as the $i$-th decomposed components of $Y^c$ by the Blaschke decomposition, and hence we have $\tilde{Y}^{c}_i(t)=e^{-i2\pi\xi_0 t}Y^{c0}_i(t)$ as the $i$-th estimated components of $Y(t)$. The estimated IF of $\phi'_i(t)$ from $\tilde{Y}_i(t)$ and $\tilde{Y}^c_i(t)$ are denoted as $\tilde{\phi}^{\prime}_{i,Y}(t)$ and $\tilde{\phi}^{\prime c}_{i,Y}(t)$ respectively.
The same notations are applied to the decomposition of $Z(t)$.

\begin{figure}[h!]
\includegraphics[width=.9\textwidth]{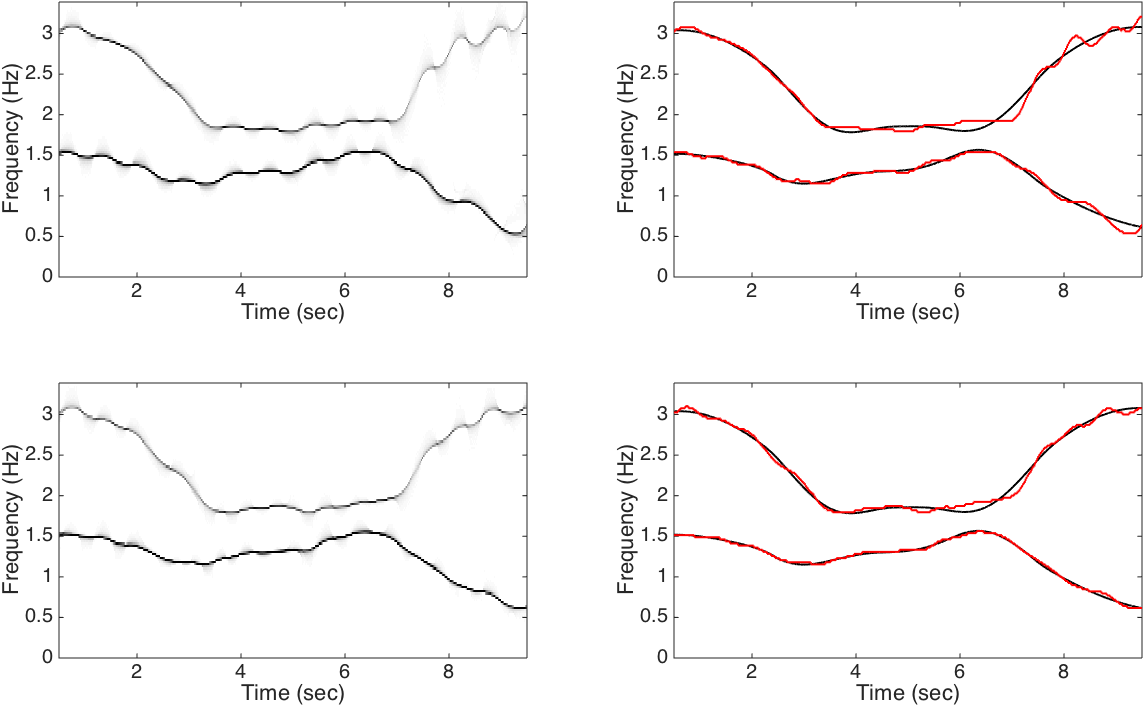}
\captionsetup{width=0.95\textwidth}
\caption{The time-frequency (TF) representation $R_Y$ is shown on the left upper subplot; the TF representation $R_Y^c$ determined from $Y^c$ is shown on the left lower subplot; right upper subplot: the true instantaneous frequency (IF) of $f_1$ and $f_2$ are shown as the black curves, and the estimated IF's from $Y$ are shown as the red curves;  right lower subplot: the true IF's of $f_1$ and $f_2$ are shown as the black curves, and the estimated IF's from $Y^c$ are shown as the red curves. We could see that the IF estimated from $Y^c$ is more accurate.}
\label{Example:5}
\end{figure}

\begin{figure}[h!]
\includegraphics[width=.9\textwidth]{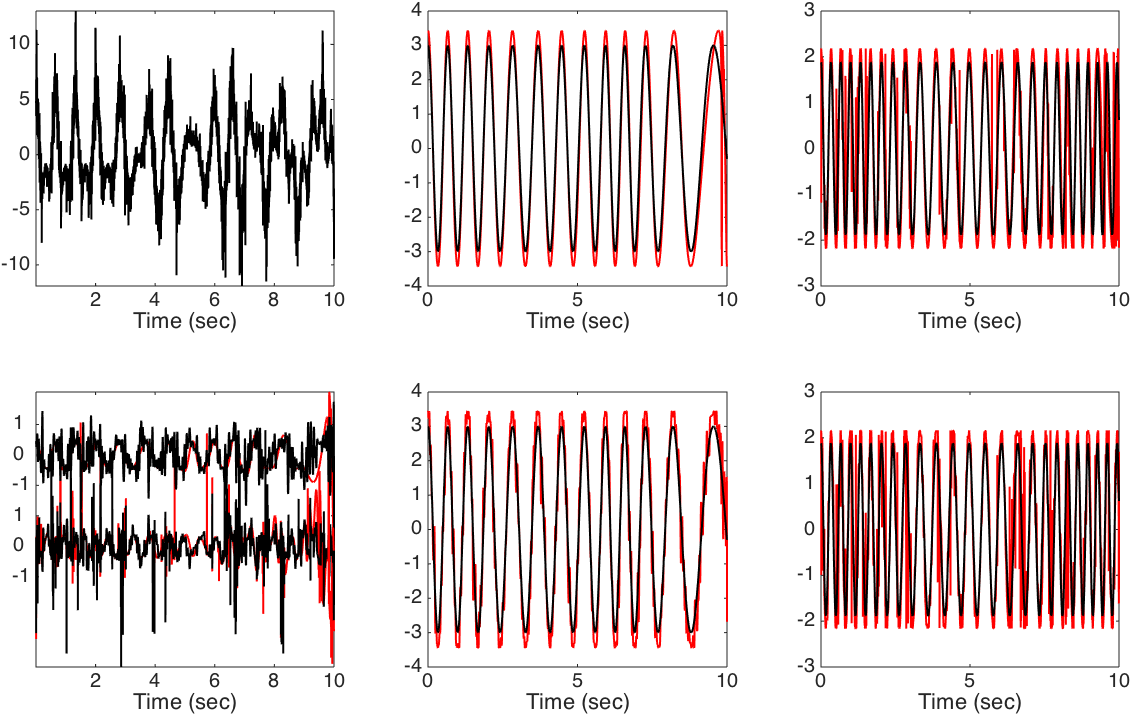}
\captionsetup{width=0.95\textwidth}
\caption{An illustration of the Blaschke decomposition of $Z(t)=f(t)e^{\sigma\xi}$. Top left: $Z$; top middle: $f_1$ is shown in black and $\tilde{Z}_1$ is shown in red; top right: $f_2$ is shown in black and $\tilde{Z}_2$ is shown in red; bottom left: the upper red curve is $f_1-\tilde{Z}_1$, the upper black curve is $f_1-\tilde{Z}^c_1$, the lower red curve is $f_1-\tilde{Z}_1$ and the lower black curve is $f_1-\tilde{Z}^c_1$; bottom middle: $f_1$ is shown in black and $\tilde{Z}^c_1$ is shown in red; bottom right: $f_2$ is shown in black and $\tilde{Z}^c_2$ is shown in red. It is clear that the carrier frequency helps to increase the decomposition accuracy.}
\label{Example:6}
\end{figure}

The decomposition results of one realization of $Y(t)$ (respectively $Z(t)$) are shown in Figures \ref{Example:4} and \ref{Example:5} (respectively Figures \ref{Example:6} and \ref{Example:7}). We could see that under the additive noise, the components obtained from the Blaschke decomposition are contaminated by some ``shot noise'', and this artifacts still exist when we apply the frequency carrier. While the decomposition is contaminated by the shot noise, due to the robustness of SST, the Blaschke TF representation is relatively clean. We could thus see the benefit of combining the Blaschke decomposition and SST. On the other hand, under the multiplicative noise, in addition to the shot noise artifact, the components obtained from the Blaschke decomposition have larger amplitude. This is caused by the amplitude perturbation induced by the multiplicative noise $e^{\xi/2}$, which has a strictly positive mean. Again, due to the robustness of SST, the Blaschke TF representation is relatively clean.

To quantify how the noise influences the result, we repeat the decomposition for 100 independent realizations of the noise, and report the 100 SER's as mean$\pm$std. The SER's for the decomposition of $Y$ are $\text{ER}(\tilde{Y}_1)=0.18\pm 0.01$, $\text{ER}(\tilde{Y}_2)=0.3\pm 0.01$, $\text{ER}(\tilde{Y}^c_1)=0.16\pm 0.01$, and $\text{ER}(\tilde{Y}^c_2)=0.19\pm 0.01$. To test if one estimator is better than the other, we apply the Mann-Whitney test and view $p<0.01$ as statistically significant. The testing results show that the carrier frequency technique does help with statistical significance. The SER's of the estimated IF are $\text{ER}(\tilde{\phi}'_{1,Y})=0.034\pm0.002$, $\text{ER}(\tilde{\phi}'_{2,Y})=0.03\pm0.002$,
 $\text{ER}(\tilde{\phi}^{\prime c}_{1,Y})=0.023\pm0.003$, and $\text{ER}(\tilde{\phi}^{\prime c}_{2,Y})=0.019\pm0.002$. The improvement of the IF estimation by the carrier frequency method is statistically significant.

The SER's for the decomposition of 100 realizations of $Z$ are $\text{ER}(\tilde{Z}_1)=0.22\pm 0.01$, $\text{ER}(\tilde{Z}_2)=0.39\pm 0.02$, $\text{ER}(\tilde{Z}^c_1)=0.23\pm 0.01$, and $\text{ER}(\tilde{Z}^c_2)=0.33\pm 0.02$. In this example,  with statistical significance, the carrier frequency technique performs better when extracting the second component, while it performs worse when extracting the first component. The SER's of the estimated IF are $\text{ER}(\tilde{\phi}'_{1,Z})=0.036\pm0.002$, $\text{ER}(\tilde{\phi}'_{2,Z})=0.031\pm0.004$,
 $\text{ER}(\tilde{\phi}^{\prime c}_{1,Z})=0.022\pm0.003$, and $\text{ER}(\tilde{\phi}^{\prime c}_{2,Z})=0.022\pm0.002$. Again, the improvement of the IF estimation by the carrier frequency method is statistically significant.

\begin{figure}[ht]
\includegraphics[width=.95\textwidth]{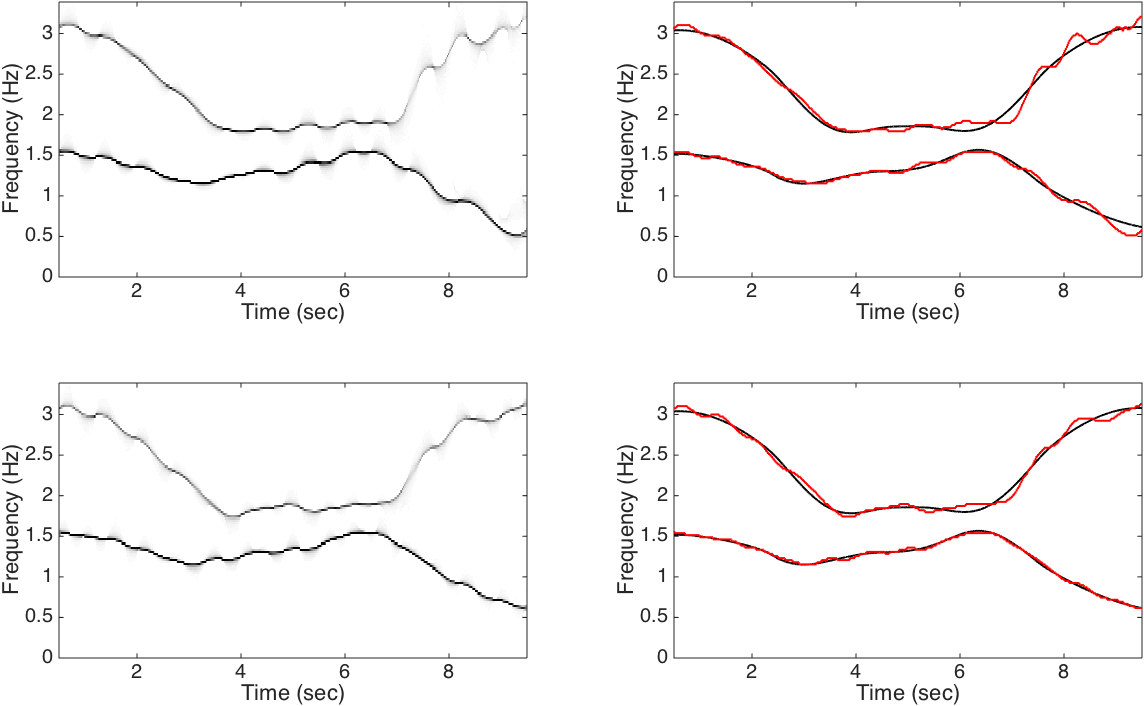}
\captionsetup{width=0.95\textwidth}
\caption{The time-frequency (TF) representation $R_Z$ is shown on the left upper subplot; the TF representation $R_Z^c$ is shown on the left lower subplot; right upper subplot: the true instantaneous frequency (IF) of $f_1$ and $f_2$ are shown as the black curves, and the estimated IF's from $Z$ are shown as the red curves;  right lower subplot: the true IF's of $f_1$ and $f_2$ are shown as the black curves, and the estimated IF's from $Z^c$ are shown as the red curves. We could see that the IF estimated from $Z^c$ is more accurate.}
\label{Example:7}
\end{figure}

\subsection{Respiratory signal analysis}

In the past decades, more and more clinical researches focus on extracting possible hidden dynamics from the respiratory signal, which are not easily observed by our naked eyes. The quantity IF has been shown to be a successful surrogate for the breathing rate variability, which helps clinicians' decision making in several problems like the ventilator weaning \cite{Wu_Hseu_Bien_Kou_Daubechies:2013}, sleep dynamics detection \cite{Wu_Talmon_Lo:2015},  etc. We now demonstrate the Blaschke decomposition result from a respiratory signal recorded from a subject when the subject is under general anesthesia. The signal lasts for 3 minutes and is sampled at 25 Hz. Since the amplitude of the respiratory signal is not constant, we take $D=5$ when we apply the Blaschke decomposition. We also apply $5$ Hz carrier frequency to analyze the data. 

The results are shown in Figure \ref{Example:Resp}. The first decomposed component well reflects how fast the flow signal oscillates; that is, the IF information of the flow signal could be obtained from it. Presumably, the second decomposed component could be the ``multiple'' of the first component, and we could see the unstable ``shot noise'' around, which might come from the inevitable noise. We could see that the Blaschke TF representation is cleaner than the SST TF representation, since the interference issue in SST is reduced by taking the Blaschke decomposition into account. \newpage

\begin{figure}[h!]
\includegraphics[width=.95\textwidth]{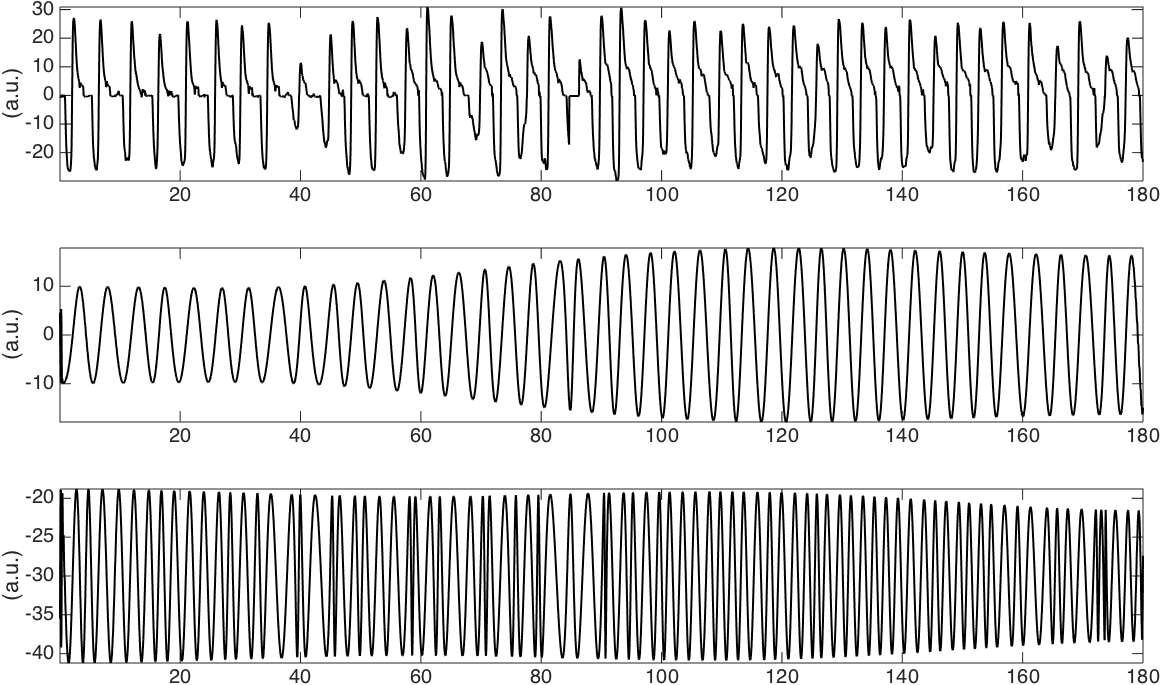}
\includegraphics[width=.95\textwidth]{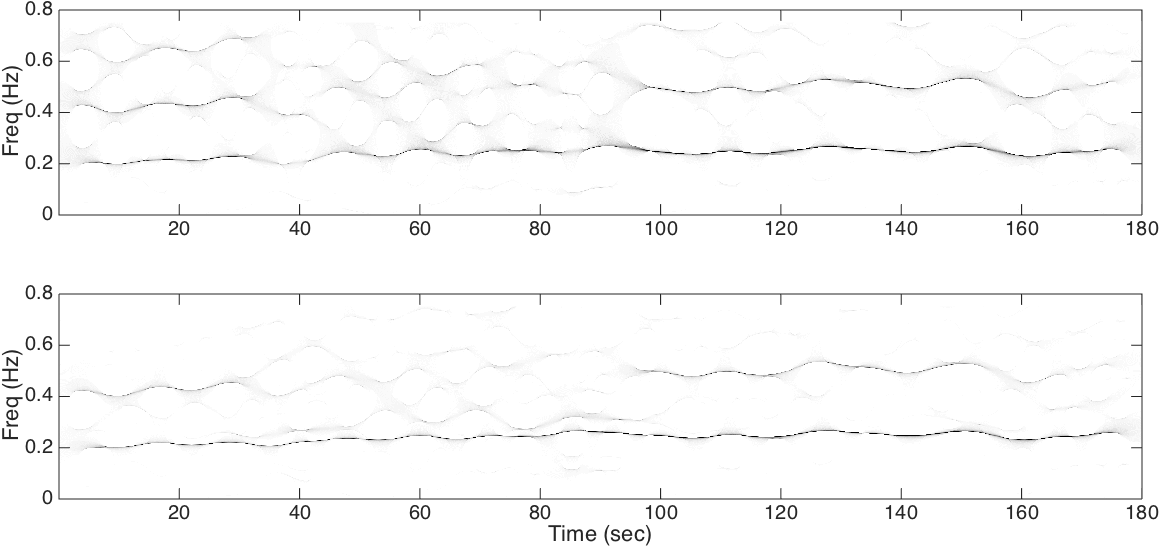}
\captionsetup{width=0.95\textwidth}
\caption{From top to bottom: the respiratory signal $f$; the first component decomposed from $f$ by the Blaschke decomposition; the second component decomposed from $f$ by the Blaschke decomposition; the SST time-frequency (TF) representation of $f$; the Blaschke TF representation of $f$.}
\label{Example:Resp}
\end{figure}

\subsection{Gravity wave example}

Gravitational waves are ripples in space-time produced by some of the most violent events in the cosmos, which is predicted by Albert Einstein in 1916 via the general theory of relativity. The gravitational-wave signal GW150914 was observed on September 14, 2015 by the two detectors of the Advanced Laser Interferometer Gravitational-wave Observatory (LIGO). GW150914 is the first direct observation of a pair of black holes merging to form a new single black hole, and how sure it is a real astrophysical event has been extensively discussed and confirmed \cite{PhysRevLett.116.061102QQ}. The signals collected from Hanford, Washington (denoted as H1 signal) and Livingston, Louisiana (denoted as L1 signal) are available for download from \url{https://losc.ligo.org/events/GW150914/}, where details of how the signals are collected, the argument why the signals are correct, and the theoretical explanation and background are also available. Both signals last for 0.21 second and are sampled at 16,384 Hz.

\begin{figure}[h!]
\includegraphics[width=.93\textwidth]{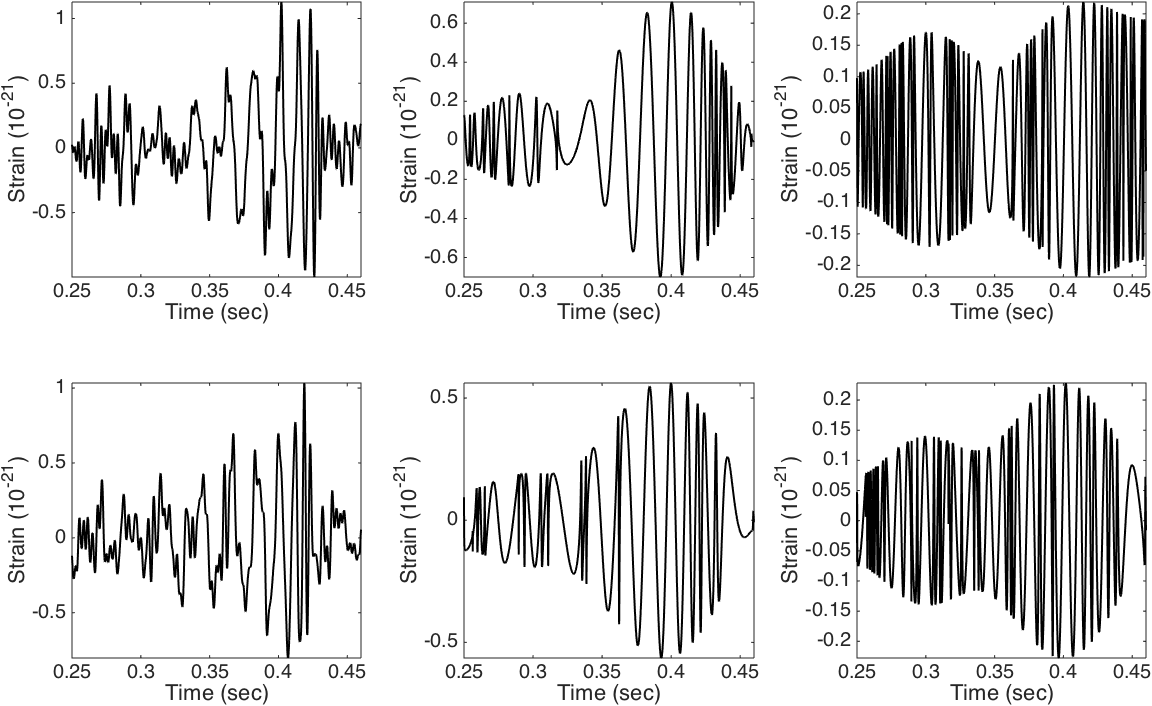}
\captionsetup{width=0.95\textwidth}
\caption{Top row: from left to right are the H1, the first decomposed component of H1, and the second decomposed component of H1; bottom row: from left to right are the L1, the first decomposed component of L1, and the second decomposed component of L1.}
\label{Example:LIGO1}
\end{figure}

\begin{figure}[h!]
\includegraphics[width=.93\textwidth]{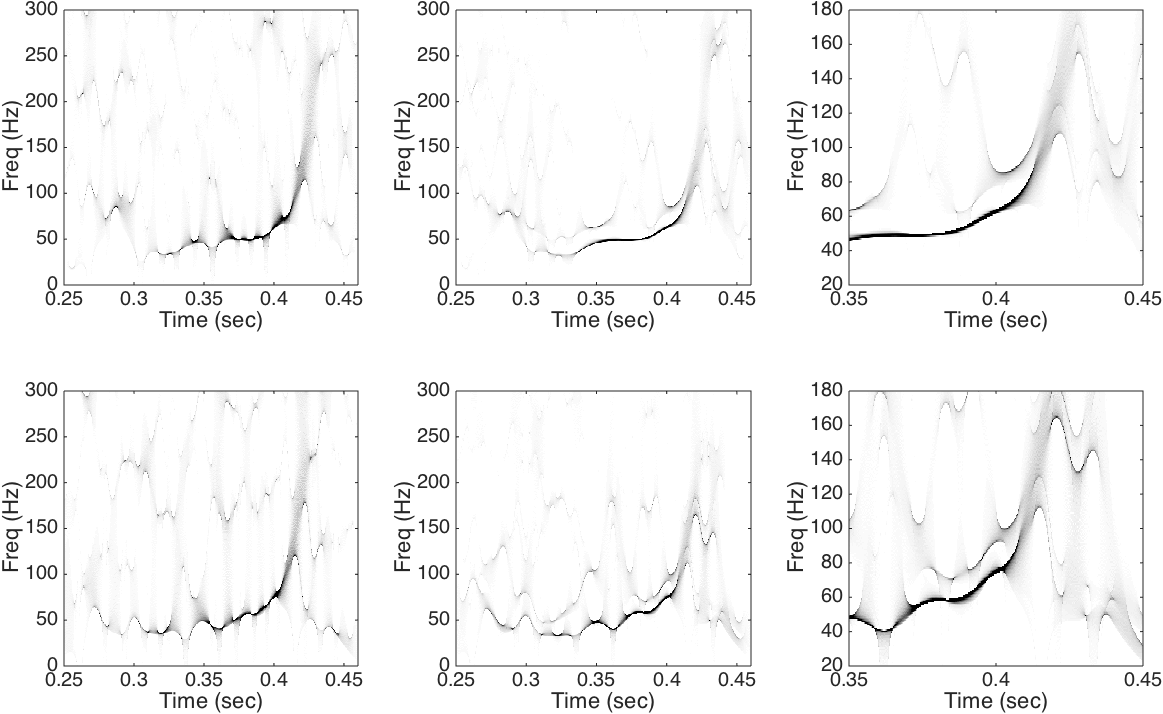}
\captionsetup{width=0.95\textwidth}
\caption{Top row: from left to right are the time-frequency (TF) representation of H1 determined by SST, the Blaschke TF representation of H1, and the zoomed in Blaschke TF representation of H1; bottom row: from left to right are the TF representation of L1 determined by SST, the Blaschke TF representation of L1, and the zoomed in Blaschke TF representation of L1.}
\label{Example:LIGO2}
\end{figure}

We now apply the Blaschke decomposition to analyze H1 and L1 signals, and show the TF representation determined by SST and the Blaschke TF representation. Since the amplitude of the signal varies, we apply the Blaschke decomposition with $D=5$. 
The decomposition results are shown in Figures \ref{Example:LIGO1}. We could see the chirp behavior after 0.35 seconds in both signals. The decomposed signals seem to have several ``phase jumps'', for example around 0.3 second of the first decomposed component of both H1 and L1. These phase jumps could be explained by the inevitable noise in the recorded signal.   
The SST TF representations and the Blaschke TF representations of of H1 and L1 are shown in \ref{Example:LIGO2}. For both H1 and L1, compared with the TF representation determined by SST, we could see that the Blaschke TF representation is sharper with less ``background speckles'' in the TF representation. Note that while the findings match the reported TF representations determined by CWT in \cite{PhysRevLett.116.061102QQ}, to further explore the results shown here and the potential of the proposed method, we need extensive collaborations with field experts, and the results will be reported in the future work.

\subsection{Detecting Roots}\label{Section:Numerics:RootDetection}
In this last subsection, we show how the Blaschke decomposition reveals the roots of a given analytic function $f:\mathbb{T}\to\mathbb{C}$, when combined with the Poisson convolution.

\begin{figure}[h!]
\includegraphics[width=.95\textwidth]{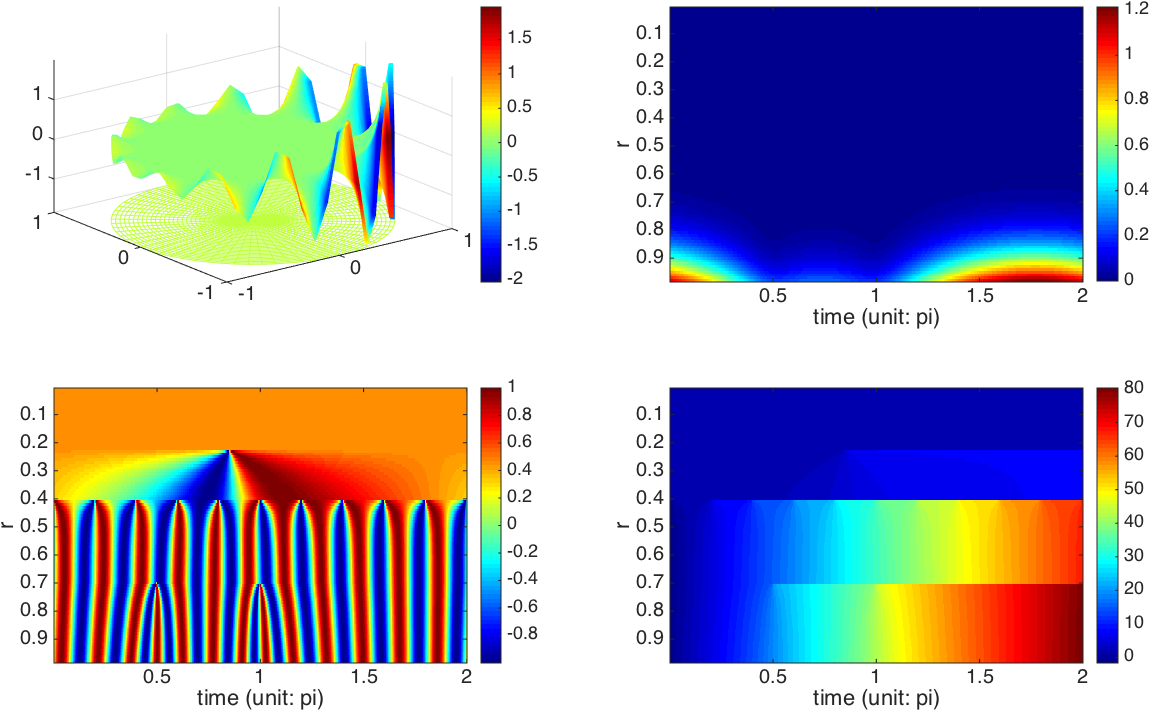}
\captionsetup{width=0.95\textwidth}
\caption{Left upper: the Poisson integral of $f$ defined by $f(e^{i\theta})=(e^{i10\theta}-0.4^{10})(e^{i\theta}-0.7i)(e^{i\theta}+0.7)(e^{i\theta}+0.1+0.2i)$, where $\theta\in\mathbb{T}$ on the unit disk. The z-axis indicates the real part of the Poisson integral and the color indicates the phase ranging from $-\pi$ to $\pi$. Right upper: the absolute value of $u_r$, where $r\in[0.01,0.99]$. Left bottom: the real part of $B_r$. Right bottom: the phase of $B_r$. Clearly it is not easy to read the root location from the absolute value of $u_r$, but the root location information can be clearly seen from reading the ``transition'' in the phase of $B_r$.}
\label{Example:Root1}
\end{figure}

 For $0\leq r<1$, denote $$u_r(\theta)=\frac{1}{2\pi}\int^{\pi}_{-\pi}\frac{1-r^2}{1-2r\cos(\theta-t)+r^2}f(e^{it})dt,$$ where $\theta\in\mathbb{T}$ and the Blaschke decomposition of $u_r=B_rG_r$. Consider $f(e^{i\theta})=(e^{i10\theta}-0.4^{10})(e^{i\theta}-0.7i)(e^{i\theta}+0.7)(e^{i\theta}+0.1+0.2i)$, where we have 10 roots on the circle of radius $0.4$, 2 roots on the circle of radius $0.7$ and one root on the circle of radius $0.224$. According to the discussion in Section \ref{Section:RootDetection}, when we apply the Blaschke decomposition, we should see a ``transition'' from $B_{r+\delta}$ to $B_{r-\delta}$, where $0<\delta\ll1$, where $r=0.7, 0.4$ and $0.224$.

We sample $1024$ points from $f$ uniformly on $\mathbb{T}$, and evaluate $u_r$ by the Poisson integral directly by evaluating the Rieman sum, where $r$ ranges from $0.01$ to $0.99$ with the uniform grid length $0.01$.
For each $u_r$, we apply the Blaschke decomposition to evaluate $B_r$. The results of the Poisson integral, the Blaschke decomposition are shown in Figure \ref{Example:Root1}. Clearly, it is not easy to directly read the root location from $u_r$. However, the root location information could be clearly seen from reading the ``transition'' in the phase of $B_r$. Particularly, when $r>0.7$ the phase of $B_r$ increases up to $81$, while the phase of $B_r$ increases only up to $70$ when $0.7>r>0.4$ and so on. The real part of $B_r$, when $r>0.7$ indicates that $B_r$ oscillates with a fast-varying IF around time $0.5$ and $1$, which reflects the roots at $0.7i$ and $-0.7$. When $0.4<r<0.7$, $B_r$ also oscillates with a non-constant IF, but the IF varies more slowly compared with the $B_r$ when $r>0.7$. This observation reflects the fact that the non-constant IF is captured by the non-zero roots, and the farther the root is to $0$, the faster the IF varies. 
\begin{figure}[h!]
\includegraphics[width=.95\textwidth]{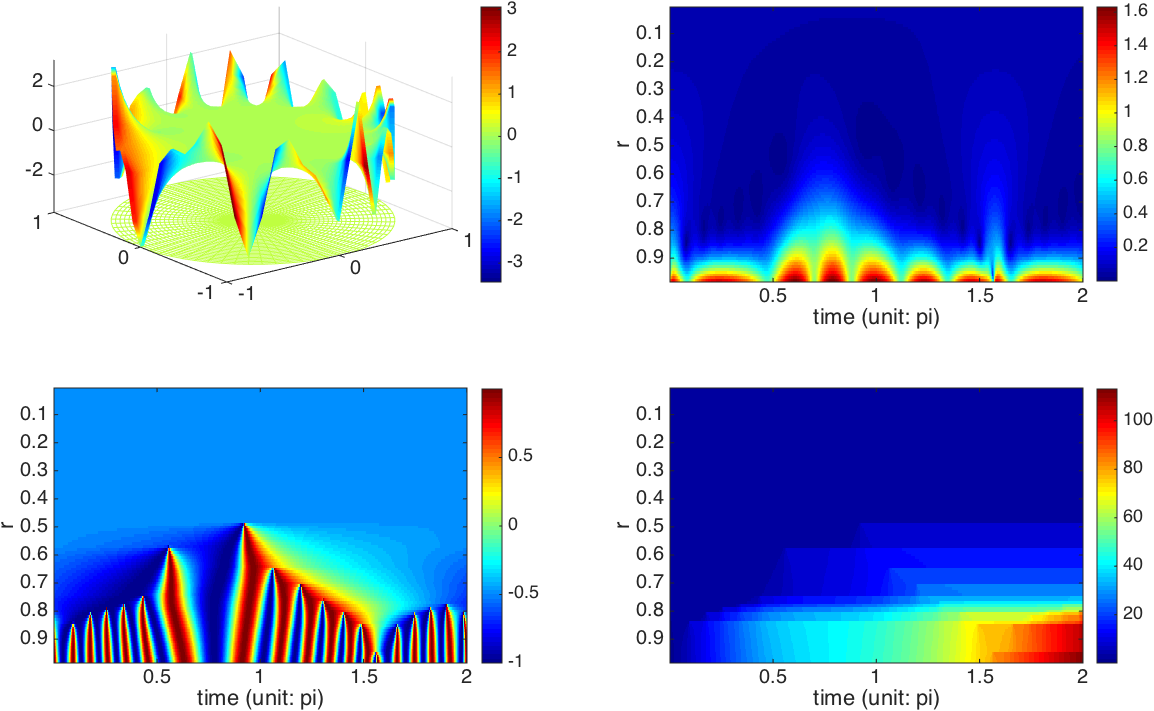}
\captionsetup{width=0.95\textwidth}
\caption{Left upper: the Poisson integral of $f$ defined in (\ref{Equation:Definiteion:Simulation:Clean}) on the unit disk. The $z$-axis indicates the real part of the Poisson integral and the color indicates the phase ranging from $-\pi$ to $\pi$. Right upper: the absolute value of $u_r$, where $r\in[0.01,0.99]$. Left bottom: the real part of $B_r$. Right bottom: the phase of $B_r$. Clearly it is not easy to read the root location from the absolute value of $u_r$, but the root location information could be clearly seen from reading the ``transition'' in the phase of $B_r$.}
\label{Example:Root2}
\end{figure}

Next, we show the result for $f(e^{i\theta})$ defined in (\ref{Equation:Definiteion:Simulation:Clean}). We sample $f$ at the sampling rate of 128Hz and $L=8$; that is, we sample $1024$ points. The results of the Poisson integral, the Blaschke decomposition are shown in Figure \ref{Example:Root2}. It is clear again that from reading the Poisson integral, little information about roots could be directly obtained. However, we could see several transitions across different $r$ in the phase plot of $B_r$, which provides roots' locations.   
While the exact relationship between roots and the analytic function satisfying the adaptive harmonic model is still open, this result shows the potential of applying the Blaschke decomposition to detect the roots of a given analytic function.
For example, the nonlinear ``strips'' in the real of $B_r$ when $r$ changes might help quantify the relationship. The study of this relationship will be reported in the future work.

\bibliography{Blaschke}
\bibliographystyle{plain}

\end{document}